\documentclass[12pt]{amsart}

\usepackage{hyperref}

\usepackage{fullpage}
\usepackage{amsfonts}
\usepackage{amsthm}
\usepackage{enumerate}
\usepackage{verbatim}
\usepackage{graphicx}

\newcommand{\R}{\mathbb{R}}
\newcommand{\C}{\mathbb{C}}
\newcommand{\N}{\mathbb{N}}

\newcommand{\Z}{\mathbb{Z}}

\newcommand{\T}{\mathbb{T}}

\newcommand{\ve}{\varepsilon}
\newcommand{\vp}{\varphi}
\newcommand{\lan}{\langle}
\newcommand{\ran}{\rangle}
\DeclareMathOperator{\spa}{span}
\DeclareMathOperator{\tr}{tr}

\DeclareMathOperator{\col}{col}
\DeclareMathOperator{\maxroot}{maxroot}

\DeclareMathOperator{\range}{range}
\DeclareMathOperator{\rank}{rank}

\newcommand{\E}[2]{\mathbb{E}_{#1}\left[ #2 \right] }
\renewcommand{\P}[2]{\mathbb{P}_{#1}\left( #2 \right)}

\newcommand{\bI}{\mathbf I}

\numberwithin{equation}{section}
\theoremstyle{plain} 

\newtheorem{theorem}{Theorem}[section]
\newtheorem*{theorem*}{Theorem}
\newtheorem{corollary}[theorem]{Corollary}
\newtheorem{lemma}[theorem]{Lemma}

\newtheorem*{problem}{Problem}

\theoremstyle{definition}
\newtheorem{definition}[theorem]{Definition}

\theoremstyle{remark}
\newtheorem{remark}[theorem]{Remark}

\begin{document}

\title{Selector form of Weaver's conjecture \\
and frame sparsification}

\author{Marcin Bownik}

\address{Department of Mathematics, University of Oregon, Eugene, OR 97403--1222, USA}
\email{mbownik@uoregon.edu}




\thanks{The author was supported in part by the NSF grants DMS-1956395 and DMS-2349756. The author wishes to thank Nikhil Srivastava for useful discussions on mixed characteristic polynomials of block diagonal matrices. The author also thanks Arthur DressenWall, Jordy Van Velthoven, Pu-Ting Yu, and anonymous referees for useful comments on the paper.}

\begin{abstract} We show an extension of a probabilistic result of Marcus, Spielman, and Srivastava \cite{MSS}, which resolved the Kadison-Singer problem, to block diagonal positive semidefinite random matrices. We use this result to show several selector results, which generalize their partition counterparts. This includes a selector form of Weaver’s KS$_r$ conjecture for block diagonal trace class operators, which extends a selector result for Bessel sequences, or equivalently rank one matrices, due to Londner and the author \cite{BL}. We also show  a selector variant of Feichtinger’s conjecture for a (possibly infinite) collection of Bessel sequences, extending earlier results for a single Bessel sequence. We prove a generalization of the $R_\ve$ conjecture of Casazza, Tremain, and Vershynin \cite{CT} for infinite collection of equal norm Bessel sequences. In particular, our selector result yields a conjectured asymptotically optimal bound for a single Bessel sequence in terms of Riesz sequence tightness parameter. 

We establish an iterated selector form of  Weaver’s KS$_2$ conjecture and show its applications. This includes a solution of an open problem on nearly unit norm Parseval frames of exponentials, which was posed by Londner and the author \cite{BL}. We generalize a discretization result  for continuous frames by Freeman and Speegle \cite{FS} in two ways.  First, we extend their result from the setting of rank one operators to positive trace operator valued measures. Second, we establish a nearly tight discretization of bounded continuous Parseval frames. In particular, our selector result yields an improvement of the result of Nitzan, Olevskii, and Ulanovskii \cite{NOU} and implies the existence of nearly tight exponential frames for unbounded sets with an explicit control on their frame redundancy.
\end{abstract}

\maketitle

\section{Introduction}

The Kadison-Singer problem \cite{KS} was known to be equivalent to several outstanding problems in analysis such as the Anderson paving conjecture \cite{AA, And, CE, CEKP}, the Bourgain-Tzafriri restricted invertibility conjecture \cite{BT1, BT2, BT3}, Feichtinger's conjecture \cite{BS, CCLV}, and Weaver's conjecture \cite{We}. 
Marcus, Spielman, and Srivastava \cite{MSS} resolved the Kadison-Singer problem by showing Weaver's KS$_r$ conjecture using the following probabilistic result on random positive semidefinite matrices. We refer to the surveys \cite{CFTW, CT} and the papers \cite{MB, CT2, MS} discussing the solution of the Kadison-Singer problem and its various ramifications. More recent developments on the Kadison-Singer problem include \cite{AW0, Bra, KLS, RL, RS, XXZ2}.

\begin{theorem}\label{thmp}
Let $\epsilon > 0$. Suppose that $X_1, \dots, X_m$ are jointly independent $d\times d$ positive semidefinite random matrices, which take finitely many values and satisfy  
\begin{equation}\label{thmp1}
\sum_{i=1}^m \E{}{X_{i} } \le \bI
\qquad\text{and}\qquad
\E{}{ \tr X_i } \leq \epsilon \quad\text{for all } i.
\end{equation}
Then,
\begin{equation}\label{thmp2}
\P{}{\bigg\| \sum_{i=1}^m X_i \bigg\| \leq (1 + \sqrt{\epsilon})^2} > 0.
\end{equation}
\end{theorem}

The original result in \cite{MSS} had a superfluous rank one assumption on random matrices $X_i$. The extension of Theorem \ref{thmp} to higher ranks was announced without the proof by Michael Cohen \cite{Cohen}, who prematurely passed away.  A proof of this result was shown independently by the author \cite{AW}. Br\"and\'en \cite[Theorem 6.1]{Bra} showed an extension of Theorem \ref{thmp} in the realm of hyperbolic polynomials with more precise bounds depending on the ranks of random matrices $X_i$, see also \cite{XXZ2}. In this paper we show a block diagonal extension of Theorem \ref{thmp}.

\begin{theorem}\label{thb}
Let $\epsilon > 0$. Suppose that $X_1, \dots, X_m$ are jointly independent $dk\times dk$ positive semidefinite random matrices, which take finitely many values and have block diagonal form
\begin{equation}\label{thb1}
X_i = \begin{bmatrix} X_i^{(1)} & & & \\
& X_i^{(2)} & & \\
& & \ddots & \\
& & & X_i^{(k)} \end{bmatrix}.
\end{equation}
Assume that $X^{(j)}_i$, $i=1,\ldots,m$, $j=1,\ldots,k$, are $d\times d$ positive semidefinite random matrices such that
\[
\sum_{i=1}^m \E{}{X_i^{(j)}} \le \mathbf I_d \qquad\text{and}\qquad \tr(X_i^{(j)}) \le \epsilon \text{ with probability 1}.
\]
Then,
\begin{equation}\label{thb2}
\P{}{\bigg\| \sum_{i=1}^m X_i \bigg\| \leq 1 + 2\sqrt{\epsilon k} + \epsilon} > 0.
\end{equation}
\end{theorem}

Since $\tr(X_i) \le k \epsilon$, one could apply Theorem \ref{thmp} directly to obtain the upper bound of $1 + 2\sqrt{\epsilon k} + \epsilon k$ in \eqref{thb2}. Hence, the improvement of Theorem \ref{thb} over Theorem \ref{thmp} lies in the absence of the factor $k$ in the $\epsilon k$ term at the cost of a slightly stronger assumption on the trace holding almost surely.
Although Theorem \ref{thb} might look like a modest improvement of Theorem \ref{thmp}, we will see that this result has far reaching consequences.

We show several applications of this result such as a selector form of Weaver’s KS$_r$ conjecture for block diagonal trace class operators. This extends a selector result for Bessel sequences, or equivalently rank one matrices, due to Londner and the author \cite[Theorem 3.3]{BL}. The block diagonal form of Weaver’s conjecture yields a simultaneous selector  KS$_r$ result for multiple Bessel families. In particular, our result implies a multi-paving result of Ravichandran and Srivastava \cite{RS}. That is, a finite collection of $m$ self-adjoint operators in $\ell^2(\N)$ with zero diagonals admits a simultaneous paving, which reduces their operator norms by a  multiplicative factor $0< \ve <1$, for some partition of size at most $O(m/\ve^2)$. 

Another application of Theorem \ref{thb} is a selector variant of Feichtinger’s conjecture for multiple Bessel sequences extending earlier results for a single Bessel sequence \cite[Theorem 2.1]{BL}. A remarkable novelty of Theorem \ref{I1} is that it applies not only to a finite, but also to an infinite collection of Bessel sequences. 

\begin{theorem}\label{I1}
Let $I$ and $K$ be countable sets.
 Suppose that  for each $j\in \N$, the system $\{u^{(j)}_i\}_{i\in I}$ is a Bessel sequence  with bound $1$, in a Hilbert space $\mathcal H_j$, such that
\begin{equation*}
\|u^{(j)}_i\|^2 \ge \epsilon_j 
\qquad\text{for all }i \in I, j\in \N.
\end{equation*}
Suppose that $\delta_0:=\sum_{j\in \N} (1-\epsilon_j) < 3/2 - \sqrt{2}$. Then, for any collection $\{ J_{k}\} _{k \in K}$ of disjoint $2$-element subsets of $I$, there exists a selector $J \subset \bigcup J_k$ satisfying 
\[
		\# (J\cap J_{k})=1\qquad\forall k
\]
such that for all $j\in \N$, $\{u^{(j)}_i\}_{i\in J}$ is a Riesz sequence with lower Riesz bound $\ge c$, where the constant $c$ depends only $\delta_0$. 
\end{theorem}

The assumption that the quantity $\delta_0$ in Theorem \ref{I1} is small is not essential. It can be relaxed by the Blaschke condition $\delta_0:=\sum_{j\in \N} (1-\epsilon_j) <\infty$. In this case we need to impose that a starting collection $\{ J_{k}\} _{k \in K}$ of disjoint subsets of $I$ have cardinalities $\ge r$, for sufficiently large $r$, which depends on $\delta_0$.

We focus on selector results for two main reasons. In general, selector results are stronger than the corresponding partition counterparts.  For example, the usual Weaver’s KS$_r$ conjecture about partitions, which was first shown in rank one case by Marcus, Spielman, and Srivastava \cite{MSS}, is implied by the corresponding selector result. The second reason is that several applications actually need selector results for a better control on the resulting frames, Riesz sequences, or trace class operators. We illustrate the versatility of selector techniques by showing:
\begin{enumerate}[a)]
\item the existence of syndetic Riesz sequences of exponentials with nearly tight bounds, see Theorem \ref{I2}, 
\item Feichtinger’s conjecture in the case when a Parseval frame is nearly unit norm, see Theorem \ref{I4}, and
\item the construction of uniformly discrete nearly tight frames of exponentials on unbounded sets, see Theorem \ref{I5}.
\end{enumerate}

We show a generalization of the $R_\ve$  conjecture of Casazza, Tremain, and Vershynin \cite[Conjecture 3.1]{CT} by showing the existence of a simultaneous selector for multiple Bessel sequences. At the same time our result gives an improved, asymptotically optimal bound for a single Bessel family in terms of Riesz sequence tightness parameter $\ve>0$, which was conjectured in \cite[Theorem 3.8]{BL}. Specializing our result to a Parseval frame of exponentials, we obtain the following theorem extending a positive density result of Bourgain and Tzafriri \cite[Theorem 2.2]{BT1} and a syndetic result of Londner and the author \cite{BL}.

\begin{theorem}\label{I2}
Given $S\subset\T=\R/\Z$ of positive measure, for any $0<\ve<1$, there exists a syndetic set 
\[
\Lambda=\{ \ldots<\lambda_{0}<\lambda_{1}<\lambda_{2}<\ldots
\} \subset\mathbb{Z}
\]
with gap 
\[
\gamma(\Lambda):= \sup_{n\in \Z} (\lambda_{n+1}-\lambda_{n})
\leq \frac{c}{|S|\ve^2}
\]
and such that $E(\Lambda)=\{ e^{2\pi i \lambda x} \} _{\lambda\in\Lambda}$  is a Riesz sequence in $L^{2}(S)$ with nearly tight bounds $(1\pm \ve)|S|$.
\end{theorem}

Superiority of selector methods over their partition counterparts is best revealed in the iterative form of the selector Weaver's KS$_2$ result. For the sake of simplicity we state it only for a single iteration. An iterated form of Theorem \ref{I3} is formulated in terms of binary selectors of higher orders. The idea of iterating KS$_2$ result originated in the work of Nitzan, Olevskii, and Ulanovskii \cite{NOU} and Freeman and Speegle \cite{FS}, see also \cite[Section 10.4]{OU}. 

\begin{theorem}\label{I3} Let $I$ be countable and $\epsilon>0$. Suppose that $\{T_i\}_{i\in I}$ is a family of positive trace class operators in a separable Hilbert space $\mathcal{H}$ satisfying 
\[
T:= \sum_{i\in I} T_i \le \bI \quad\text{and}\quad \tr(T_i)\le \epsilon \qquad\text{for all }i \in I.
\]
Then, for any partition $\{ J_{k}\} _{k \in K}$ of $I$ into $2$-element sets, there exists a selector $J$
such that 
\begin{equation*}
\bigg\| \sum_{i\in  J} T_i - \frac12 T \bigg\|, 
\bigg\| \sum_{i\in  I \setminus J} T_i - \frac12 T \bigg\|
 \le 2  \sqrt{\epsilon} + \epsilon.
\end{equation*}
\end{theorem}

 Theorem \ref{I3} has three main advantages over earlier results. The first advantage is that in contrast to the original sparsification results, where no control on resulting partitions was present, one can show the existence of nicely separated binary selectors. As an application we deduce a variant of Feichtinger's conjecture in the case when a Parseval frame is nearly unit norm. That is, the quantity $\delta_0=\sum_{j\in N} (1-\epsilon_j)$ in Theorem \ref{I1} is very small. 

To motivate this result recall a basic fact in frame theory stating that any Parseval frame $\{u_i\}_{i\in I}$ in a Hilbert space $\mathcal H$ satisfies $||u_i|| \le 1$ for all $i\in I$. Moreover, if $||u_i||=1$ for all $i\in I$, then $\{u_i\}_{i\in I}$ is necessarily an orthonormal basis of $\mathcal H$. We show that any Parseval frame $\{u_i\}_{i\in I}$ satisfying $||u_i|| \ge \ve$ for all $i\in I$, when $\ve <1$ is close to $1$, becomes a Riesz sequence after removing a small portion of vectors $u_i$, $i\in I$. How small depends on the proximity of $\ve$ to $1$. In particular, we solve an open problem on nearly unit norm Parseval frames of exponentials,  which was  posed by Londner and the author in \cite[Open Problem 2]{BL}.

\begin{theorem}\label{I4}
For any measurable subset $S\subset\T =\R/ \Z$ with positive measure, there exists a subset $\Lambda\subset\Z$ satisfying:
\begin{itemize}
\item $E(\Lambda)=\{ e^{2\pi i \lambda x } \} _{\lambda\in\Lambda}$ is a Riesz sequence in $L^{2}(S)$, and 
\item $\Lambda^{c}=\Z \setminus \Lambda$ is a uniformly discrete set satisfying 
\begin{equation*}
		\inf_{\lambda,\mu\in\Lambda^{c},\lambda\neq\mu} |\lambda-\mu|\geq\frac{C}{| \T \setminus S|}
\end{equation*}
\end{itemize}
\end{theorem}

The second advantage of Theorem \ref{I3} is that it allows a sparsification of a frame operator corresponding to Bessel sequences. Consequently, it yields a more efficient proof of the discretization problem for continuous frames, which was posed by Ali, Antoine, and Gazeau \cite{aag2} and resolved by Freeman and Speegle \cite{FS}. We improve their result by showing nearly tight discretization of continuous Parseval frames, which was an open problem unresolved in \cite{FS}. 

We show that a bounded continuous Parseval frame can be sampled to obtain a discrete frame which is nearly tight. That is, the ratio of the upper and lower bounds can be made arbitrarily close to $1$, whereas existing techniques could only guarantee this ratio to be close to $2$. In addition, our selector result yields uniformly discrete sampling sets.
In particular, we obtain the following improvement of  the result of Nitzan, Olevskii, and Ulanovskii on exponential frames for unbounded sets \cite{NOU}.

\begin{theorem} \label{I5}
There exist constants $c_0, c_1>0$ such that for any set $S \subset \R$ of finite positive measure and $0<\ve<1$, there exists a set $\Lambda \subset \R$ satisfying:
\begin{itemize}
\item $E(\Lambda)=\{ e^{2\pi i \lambda x } \} _{\lambda\in\Lambda}$ is a frame in $L^2(S)$ with nearly tight frame bounds $(1\pm \ve)a|S|/\ve^2$, where $c_0/2\le a \le c_0$,
\item
$\Lambda$ is a uniformly discrete set 
\[
		\inf_{\lambda,\mu\in\Lambda,\lambda\neq\mu} |\lambda-\mu|\geq c_1 \frac{\ve^2}{ |S|}.
\]
\end{itemize}
\end{theorem}

In contrast to Theorem \ref{I5}, the corresponding problem of existence of Riesz bases of exponentials in $L^2(S)$ is challenging. Riesz bases for arbitrary finite unions of intervals were constructed by Kozma and Nitzan \cite{kn}. In higher dimensions, Riesz bases for convex, symmetric
polygons were constructed by Debernardi and Lev \cite{dl}. For a construction of Riesz bases for multi-tiling, see  \cite{gl, kol}. In a stunning recent work  \cite{kno}, Kozma, Nitzan, and Olevskii have shown the first negative result. There exists a bounded set $S \subset \R$ of positive measure such that no subset $\Lambda \subset \R$ leads to a Riesz basis of exponentials $E(\Lambda)$ for $L^2(S)$.

The final advantage of Theorem \ref{I3} is that it applies not only to rank one operators, but also to trace class operators. As an application we show a discretization result for positive trace operator valued measures. Continuous frames are merely a special class of such measures corresponding to rank one operators. Hence, this generalizes the discretization result of Freeman and Speegle \cite{FS} in yet another direction. For a related recent work on sampling discretization in $L^2$, which also relies on the solution of the Kadison-Singer problem \cite{MSS}, we refer to \cite{sam1, sam2, sam3}.

The paper is organized as follows. In Section \ref{S2} we develop several results on mixed characteristic polynomials and
 prove Theorem \ref{thb}. In Section \ref{S3} we show a selector form of Weaver’s KS$_r$ conjecture for block diagonal trace class operators. As its application we recover a multi-paving result of Ravichandran and Srivastava \cite{RS}. 
 In Section \ref{S4} we show  a selector variant of Feichtinger’s conjecture for a (possibly infinite) collection of Bessel sequences, which includes Theorem \ref{I1} as its special case. 
We also show a selector generalization of the $R_\epsilon$ conjecture. 
 In Section \ref{S5} we prove a binary selector form of KS$_2$ result by iterative applications of Theorem \ref{I3}. In Section \ref{S6} we show a selector form of Feichtinger’s conjecture for nearly unit norm Parseval frames. As an application  to exponential frames we deduce Theorem \ref{I4}. In Section \ref{S7} we prove discretization results for continuous frames and their extension for positive trace operator valued measures. 
Finally, in Section \ref{S8} we explore applications of selector results to exponential frames by showing Theorems \ref{I2}, \ref{I4}, and \ref{I5}.

\section{Block diagonal random matrices}\label{S2}

In this section we show several results on a mixed characteristic polynomial, which was introduced by Marcus, Spielman, and Srivastava \cite{MSS} in their solution of the Kadison-Singer problem. Our results focus on the largest root of mixed characteristic polynomial and its behavior under perturbations by positive semidefinite matrices. In particular, we show an estimate on the largest root of a mixed characteristic polynomial between a collection of block diagonal matrices and its individual blocks. Based on these results we prove a block diagonal extension of Theorem \ref{thmp}. The assumption that blocks in Theorem \ref{thb} have the same size is not essential. It is made only to simplify the statement. More importantly, we also relax the assumption on equal traces of each block by showing the following variant of Theorem \ref{thb}.

\begin{theorem}\label{thbx}
Let $\epsilon_1,\ldots,\epsilon_k > 0$. Suppose that $X_1, \dots, X_m$ are jointly independent $dk\times dk$ positive semidefinite random matrices, which take finitely many values, in the block diagonal form \eqref{thb1}.
Assume that each block $X^{(j)}_i$, $i=1,\ldots,m$, $j=1,\ldots,k$, is a $d\times d$ positive semidefinite random matrix such that
\begin{equation}\label{thbx1}
\sum_{i=1}^m \E{}{X_i^{(j)}} \le \mathbf I_d \qquad\text{and}\qquad \tr(X_i^{(j)}) \le \epsilon_j \text{ with probability 1}.
\end{equation}
Then,
\begin{equation}\label{thbx2}
\P{}{ \forall j \in [k] \ \bigg\| \sum_{i=1}^m X_i^{(j)} \bigg\| \leq 1 + 2\bigg(\sum_{l=1}^k \epsilon_l \bigg)^{1/2} + \epsilon_j} > 0.
\end{equation}
\end{theorem}

We start by recording known results about mixed characteristic polynomials.

\begin{definition} \label{dmcp}
Let $A_1, \ldots, A_m$ be $d\times d$ matrices. The mixed characteristic polynomial is defined for $z\in \C$ by
\[
\mu[A_1,\ldots,A_m](z) =  \bigg(\prod_{i=1}^m (1 - \partial_{z_i}) \bigg) \det \bigg( z \mathbf I + \sum_{i=1}^m z_i A_i \bigg)\bigg|_{z_1=\ldots=z_m=0}.
\]
\end{definition}

An important property of a mixed characteristic polynomial is multi-affinity with respect to each argument. The proof of this result can be found in \cite[Lemma 3.2]{MB}.

\begin{lemma}\label{mas} For a fixed $z\in \C$, the mixed characteristic polynomial mapping 
\[
\mu: M_{d\times d}(\C) \times \ldots\times  M_{d \times d}(\C) \to \C
\]
 is multi-affine and symmetric. That is, $\mu$ is affine in each variable and its value is the same  for any permutation of its arguments $A_1,\ldots,A_m$.
\end{lemma}

The following well-known lemma about real-rooted polynomials plays an essential role in the concept of interlacing family of polynomials introduced  by Marcus, Spielman, and Srivastava in \cite{MSS0, MSS}. 

\begin{lemma}\label{ell}
Let $p_1,\ldots, p_n \in \R[x]$ be real-rooted monic polynomials of the same degree. Suppose that every convex combination \[
\sum_{i=1}^n t_i p_i,\qquad{where } \ \sum_{i=1}^n t_i =1,\ t_i\ge 0
\]
is  a real-rooted polynomial. Then, for any such convex combination there exist $1\le i_0, j_0 \le n$ such that
\begin{equation}\label{ell1}
\maxroot(p_{i_0}) \le \maxroot\bigg(\sum_{i=1}^n t_i p_i \bigg)
\le \maxroot(p_{j_0}).
\end{equation}
\end{lemma}

In this work we will not make a direct use of interlacing families of polynomials, but instead we employ its consequence in the form of the following lemma, see  \cite[Lemma 2.17]{AW}.

\begin{lemma}\label{max}
Suppose that $X_1, \dots, X_m$ are jointly independent random positive semidefinite $d\times d$ matrices which take finitely many values.  Then with positive probability
\begin{equation}\label{max1}
\maxroot(\mu[X_1,\ldots, X_m]) \le  \maxroot(\mu[\E{}{X_1}, \ldots, \E{}{X_m}]).
\end{equation}
\end{lemma}

The monotonicity property of the maximal root of a mixed characteristic polynomial was shown by the author \cite[Lemma 2.18]{AW}.

\begin{lemma}\label{mono}
Let $A_1,\ldots,A_m$ and $B_1, \ldots, B_m$ be positive semidefinite hermitian $d\times d$ matrices such that $A_i \le B_i$ for  $i=1,\ldots,m$.
Then,
\begin{equation}\label{norm2}
\maxroot({\mu[ A_1, \ldots, A_m]}) \le  \maxroot(\mu[ B_1,\ldots, B_m]).
\end{equation}
\end{lemma}

For any $d\times d$ matrix $A$, the mixed characteristic polynomial of $A$ is given by the explicit formula $\mu[A](z)=z^{d-1}(z-\tr A)$, $z\in \C$. Hence, if $A$ is a positive semidefinite hermitian matrix, then
\begin{equation}\label{norm6}
\maxroot \mu[A]= \tr A.
\end{equation}
When more arguments are present in a mixed characteristic polynomial, we have the following useful bound, which is a special case of \cite[Lemma 2.24]{AW} or \cite[Theorem 5.2]{Bra}.

\begin{lemma}\label{norm}
If $A_1,\ldots,A_m$ are positive semidefinite hermitian $d\times d$ matrices, then
\begin{equation}\label{norm3}
\bigg\| \sum_{i=1}^m A_i \bigg\| \le \maxroot( \mu[ A_1,\ldots, A_m] ).
\end{equation}
\end{lemma}

We also recall a result of Marcus, Spielman, and Srivastava \cite[Theorem 5.1]{MSS} showing the upper bound on the roots of a mixed characteristic polynomial.

\begin{theorem}\label{mixed}
Let $\epsilon>0$. Suppose $A_1, \ldots, A_m$ are $d \times d$ positive semidefinite matrices satisfying
\begin{equation}\label{mixed1}
\sum_{i=1}^m A_i \le \bI 
\qquad\text{and}\qquad
\tr(A_i) \le \epsilon \quad\text{for all i}.
\end{equation}
Then, all roots of the mixed characteristic polynomial $\mu[A_1,\ldots, A_m]$ are real and the largest root
is at most $(1 + \sqrt{\epsilon})^2$.
\end{theorem}

To prove Theorem \ref{thbx} we also need to develop several new properties of mixed characteristic polynomials. We start from the formula representing a mixed characteristic polynomial in terms of mixed discriminants due to Marcus, Spielman, and Srivastava \cite[Section 7.2]{MSS}, see also \cite[Section 2]{BCMS}.

\begin{definition} For $d\times d$ matrices $A_1, \ldots, A_d$ define mixed discriminant as
\[
D(A_1,\ldots,A_d) = \bigg(\prod_{i=1}^d \partial_{z_i} \bigg) \det \bigg( \sum_{i=1}^d z_i A_i \bigg)\bigg|_{z_1=\ldots=z_d=0}.
\]
If $k<d$, we extend the definition of mixed discriminant by setting
\[
D(A_1,\ldots,A_k) = \frac{D(A_1,\ldots,A_k, \mathbf I_d[d-k])}{(d-k)!}, 
\]
where the notation $A[k]$ represents a matrix $A$ repeated $k$ times.
\end{definition}

Equivalently, a mixed discriminant of $d\times d$ matrices $A_1,\ldots, A_d$ can be computed as
\begin{equation}\label{deb}
D(A_1, \ldots, A_d) 
=
\sum_{\sigma \in S_d} \det \begin{bmatrix} \col_{\sigma(1)}(A_1) | \ldots | \col_{\sigma(d)}(A_d) 
\end{bmatrix},
\end{equation}
where $\col_j(A)$ is column $j$ of a matrix $A$ and $S_d$ is a symmetric group on $d$ elements.

\begin{lemma} Let $A_1,\ldots, A_m$ be $d\times d$ matrices. Then, for any $z\in \C$ we have
\begin{equation}\label{mdf}
\mu[A_1,\ldots,A_m](z) = \sum_{k=0}^d z^{d-k} (-1)^{k} 
\sum_{S \subset [m], \ |S|=k} D((A_i)_{i\in S}),
\end{equation}
where the inner sum is over all subsets $S$ of $[m]$ of size $k$.
\end{lemma}

\begin{proof}
The proof uses known properties of mixed discriminant, see \cite{BCMS}. By the multilinearity and symmetry of mixed discriminants we have
\[
\begin{aligned}
&\det \bigg( z\mathbf I + \sum_{i=1}^m z_i A_i \bigg) = \frac{1}{d!} D\bigg( \bigg(z \mathbf I + \sum_{i=1}^m z_a A_i \bigg) [d] \bigg)
\\
&= \frac{1}{d!} \sum_{n_0+n_1+ \ldots+ n_m=d} \frac{d!}{n_0!n_1!\ldots n_m!}z^{n_0} (z_1)^{n_1} \ldots (z_m)^{n_m} D( \mathbf I[n_0],A_1[n_1],\ldots, A_m[n_m])
\\
& = \sum_{k=0}^d z^{d-k} \sum_{n_1+ \ldots+ n_m=k}
\frac{(z_1)^{n_1} \ldots (z_m)^{n_m}}{n_1!\ldots n_m!}
D( A_1[n_1],\ldots, A_m[n_m])
\end{aligned}
\]
where the above sums are taken over non-negative $n_0,\ldots,n_m$. The partial differential operator 
\[
\C[z,z_1,\ldots,z_m] \ni P \mapsto \bigg(\prod_{i=1}^m(1-\partial_{z_i})\bigg) P(z,0,\ldots,0) \in \C[z]
\]
retains only affine portion of a polynomial $P$. That is, the above operation produces a nonzero contribution only for linear combinations of monomials $z^{n_0} (z_1)^{n_1} \ldots (z_m)^{n_m}$ with all $n_1,\ldots, n_m$ equal either to $0$ or $1$. Hence, we have
\[
\mu[A_1,\ldots,A_m](z) = 
\sum_{k=0}^d z^{d-k} \sum_{n_1+ \ldots+ n_m=k}
\frac{(-1)^{n_1} \ldots (-1)^{n_m}}{n_1!\ldots n_m!}
D( A_1[n_1],\ldots, A_m[n_m]),
\]
where the above sum is taken over $n_1,\ldots,n_m=0,1$. This proves \eqref{mdf}.
\end{proof}

A mixed characteristic polynomial of $m$ matrices can have at most $m$ nonzero roots. Hence, it makes sense to define a reduced mixed characteristic polynomial of $d\times d$ matrices $A_1,\ldots, A_m$, where $m\le d$, as
\begin{equation}\label{rmcp}
\tilde \mu[A_1, \ldots,A_m ](z) = \frac{\mu[A_1,\ldots,A_m](z)}{z^{d-m}} = \sum_{S \subset [m]} z^{m-|S|} (-1)^{|S|} D((A_i)_{i\in S}).
\end{equation}

\begin{lemma}\label{shift}
Let $A_1, \ldots, A_m$ be $d\times d$ matrices, where $m\le d$. For $i=1,\ldots, m$, let $E_i$ be an $m \times m$ matrix with all zero entries except diagonal entry $(i,i)$, which is equal to $1$. Then, for any $z\in \C$ and $\delta>0$,
\begin{equation}\label{shift0}
\tilde \mu\left[ \begin{bmatrix} \delta E_1 & \\ & A_1 \end{bmatrix}, \ldots, \begin{bmatrix} \delta E_m & \\ & A_m \end{bmatrix} \right] (z) =
\tilde \mu[A_1,\ldots,A_m](z-\delta).
\end{equation}
\end{lemma}

Note that block diagonal matrices on the left of \eqref{shift0} have size $(d+m) \times (d+m)$, whereas those on the right are $d\times d$ matrices. This necessitates a reduction of degrees of mixed characteristic polynomials in the formula \eqref{rmcp}.

\begin{proof}
By \eqref{deb} for any $S \subset [m]$,
\[
\begin{aligned}
 D\bigg( \bigg(\begin{bmatrix} \delta E_i & \\ & A_i \end{bmatrix}\bigg)_{i\in S} \bigg) &=
 \frac{1}{(d+m-|S|)!} D\bigg( \mathbf I_{d+m}[d+m-|S|], \bigg(\begin{bmatrix} \delta E_i & \\ & A_i \end{bmatrix}\bigg)_{i\in S} \bigg) \\
 &= \sum_{T \subset S} \frac{\delta^{|S|-|T|}}{(d-|T|)!}  D(\mathbf I_d[d-|T|], (A_i)_{i\in T} )
 =  \sum_{T \subset S} \delta^{|S|-|T|} D((A_i)_{i\in T} ),
\end{aligned}
\]
where the sum is taken over all subsets $T \subset S$. To justify the middle equality we employ \eqref{deb} and parametrize permutations $\sigma \in S_{d+m}$ by subsets $T \subset S$ in such a way that we choose columns in the first block $\delta E_i$ if $i \in S \setminus T$. This necessitates a unique choice of columns containing entry $\delta$ to get a nonzero contribution. Otherwise, we choose columns corresponding to the second block $A_i$ if $i \in T$. 

Then, by the binomial formula
\[
\begin{aligned}
 \sum_{S \subset [m]} z^{m-|S|} (-1)^{|S|} D\bigg( \bigg(\begin{bmatrix} \delta E_i & \\ & A_i \end{bmatrix}\bigg)_{i\in S} & \bigg)
= \sum_{S \subset [m]} z^{m-|S|} (-1)^{|S|} \sum_{T \subset S} \delta^{|S|-|T|} D((A_i)_{i\in T} )
\\
& = \sum_{T \subset [m]} D((A_i)_{i\in T} ) \sum_{S \supset T} z^{m-|S|} (-1)^{|S|} \delta^{|S|-|T|}
\\
&=\sum_{T \subset [m]} D((A_i)_{i\in T} ) (-1)^{|T|}  \sum_{S' \subset [m] \setminus T} z^{m-|T|-|S'|} (- \delta)^{|S'|}
\\
&=\sum_{T \subset [m]} D((A_i)_{i\in T} ) (-1)^{|T|}  (z-\delta)^{m-|T|}.
\end{aligned}
\]
By \eqref{rmcp} this proves \eqref{shift0}.
\end{proof}

The following lemma shows that mixed characteristic polynomial of two mutually orthogonal collections of matrices is essentially a product of their corresponding mixed characteristic polynomials.

\begin{lemma}\label{or}
 Let $A_1, \ldots, A_m$ and $B_1,\ldots, B_n$ be $d \times d$ hermitian matrices such that
\[
\range(A_i) \perp \range(B_j) \qquad\text{for all }i=1,\ldots,m, \ j=1,\ldots,n.
\]
Then, for any $z\in \C$,
\begin{equation}\label{or2}
\mu[A_1,\ldots,A_m,B_1,\ldots,B_n](z) = z^{-d} \mu[A_1,\ldots,A_m](z) \mu[B_1,\ldots, B_n](z).
\end{equation}
\end{lemma}

\begin{proof}
A direct calculation shows that for any $d\times d$ unitary matrix $U$ we have
\begin{equation}\label{uau}
\mu[UA_1U^*,\ldots,UA_mU^*](z) = \mu[A_1,\ldots,A_m](z) \qquad z\in \C.
\end{equation}
Consequently, we can assume that matrices $A_i$ and $B_j$ have block diagonal form
\[
A_i = \begin{bmatrix} A'_i & \\ & \mathbf 0_{d-k} \end{bmatrix},
\qquad
B_j = \begin{bmatrix} \mathbf 0_{k} & \\ & B'_j \end{bmatrix},
\qquad i=1,\ldots,m,\ j=1,\ldots,n,
\]
for some $k\times k$ matrices $A_1',\ldots,A_m'$ and $(d-k) \times (d-k)$ matrices $B_1',\ldots,B'_n$. Hence,
\[
\begin{aligned}
&\mu[A_1,\ldots,A_m,B_1,\ldots,B_n](z)
\\
& = \bigg(\prod_{i=1}^m (1-\partial_{z_i}) \prod_{j=1}^n (1-\partial_{w_j}) \bigg) \det\bigg( z \mathbf I_k + \sum_{i=1}^m z_i A'_i\bigg) \det\bigg( z \mathbf I_{d-k} + \sum_{j=1}^n w_j B'_j\bigg)\bigg|_{\genfrac{}{}{0pt}{}{z_1=\ldots=z_m=0}{w_1=\ldots=w_n=0}}
\\
& = \mu[A'_1,\ldots,A'_m](z) \mu[B'_1,\ldots, B'_n](z)
=  \frac{\mu[A_1,\ldots,A_m](z)}{z^{d-k}} \frac{\mu[B_1,\ldots, B_n](z)}{z^k}.
\end{aligned}
\]
This proves \eqref{or2}.
\end{proof}

The following lemma describes how the largest root of mixed characteristic polynomial behaves under perturbations by positive semidefinite matrices.

\begin{lemma} \label{zl}
 Let $A_1, \ldots, A_m$ and $Z_0$ be $d \times d$ positive semidefinite  matrices such that
\begin{equation}\label{zl2}
\range(Z_0) \perp \range(A_j) \qquad\text{for all }j=1,\ldots,m.
\end{equation}
Then, for any $d\times d$ positive semidefinite matrix $Z$ with $\tr(Z) \ge \tr(Z_0)$ we have
\begin{equation}\label{zl0}
\maxroot \mu [A_1+Z, A_2, \ldots, A_m ] \ge \maxroot \mu[A_1+Z_0,A_2,\ldots,A_m].
\end{equation}
\end{lemma}

\begin{proof}
If $Z_0$ is a zero matrix $\mathbf 0$, then the conclusion follows immediately from the monotonicity property in Lemma \ref{mono}. Hence, we can assume that $Z_0 \ne \mathbf 0$.
Next we shall show \eqref{zl0} when $A_1= \mathbf 0$.
By  Lemma \ref{mono} and \eqref{norm6} we have
\[
\begin{aligned}
\maxroot \mu[Z,A_2,\ldots,A_m] &\ge \max(\maxroot \mu[Z,\mathbf 0,\ldots,\mathbf 0],\maxroot \mu[\mathbf 0,A_2,\ldots,A_m])
\\
&=
\max(\tr(Z), \maxroot \mu[A_2,\ldots,A_m]).
\end{aligned}
\]
On the other hand, by 
Lemma \ref{or} and \eqref{norm6}
\[
\begin{aligned}
\maxroot \mu[Z_0,A_2,\ldots,A_m] &= \maxroot (\mu[Z_0]\mu[A_2,\ldots,A_m]) \\
&= \max (\tr (Z_0), \maxroot \mu[A_2,\ldots,A_m]).
\end{aligned}
\]
Hence, \eqref{zl0} holds when $A_1=\mathbf 0$. If $A_1$ is nonzero matrix, then we can find $0<p<1$ such that
\begin{equation}\label{zl6}
\maxroot \mu[\tfrac1p A_1,A_2,\ldots,A_m] = \maxroot \mu[\tfrac{1}{1-p} Z_0,A_2,\ldots,A_m].
\end{equation}
Indeed, by Lemma \ref{mono} 
\[
\lim_{p\to 0^+}  \maxroot \mu[\tfrac1p A_1,A_2,\ldots,A_m] = \infty\quad\text{ and }\quad
\lim_{p\to 1^-} \maxroot \mu[\tfrac{1}{1-p}Z_0,A_2,\ldots,A_m] = \infty.
\]
Both sides of \eqref{zl6} are continuous functions of $p\in (0,1]$ and $p\in [0,1)$, respectively. Hence, the intermediate value theorem yields \eqref{zl6}.
By the multi-affinity property in Lemma \ref{mas} for any $z\in\C$
\[
\mu[A_1+Z_0,A_2,\ldots,A_m](z)=p \mu[\tfrac1p A_1,A_2,\ldots,A_m](z)+(1-p) \mu[\tfrac{1}{1-p}Z_0,A_2,\ldots,A_m](z).
\]
Hence, \eqref{zl6} implies that
\begin{equation}\label{zl8}
\maxroot \mu[A_1+ Z_0,A_2,\ldots,A_m]=
\maxroot \mu[\tfrac1p A_1,A_2,\ldots,A_m].
\end{equation}
Since for any $z\in \C$,
\[
\mu[A_1+Z,A_2,\ldots,A_m](z)=p \mu[\tfrac1p A_1,A_2,\ldots,A_m](z)+(1-p) \mu[\tfrac{1}{1-p}Z,A_2,\ldots,A_m](z),
\]
Lemma \ref{ell} implies that either
\[
\maxroot \mu[\tfrac1p A_1,A_2,\ldots,A_m] \le \maxroot \mu[A_1+Z,A_2,\ldots,A_m]
\]
or
\[
\maxroot \mu[\tfrac{1}{1-p}Z,A_2,\ldots,A_m] \le 
\maxroot \mu[A_1+Z,A_2,\ldots,A_m].
\]
In the former case \eqref{zl8} yields \eqref{zl0}. The latter case follows from \eqref{zl6}, \eqref{zl8}, and already shown special case
\[
\maxroot \mu[\tfrac{1}{1-p}Z_0,A_2,\ldots,A_m] \le 
\maxroot \mu[\tfrac{1}{1-p}Z,A_2,\ldots,A_m]
\]
Hence, either case yields the required conclusion \eqref{zl0}.
\end{proof}

We are now ready to prove the key estimate for the largest root of mixed characteristic polynomial for a collection of block diagonal matrices.

\begin{theorem}\label{t1}
Let $d,k,m\in \N$ and $\epsilon_1,\ldots,\epsilon_k>0$. Let $A_1, \ldots, A_m$ be $dk \times dk$ positive semidefinite block diagonal matrices of the form
\begin{equation}\label{t2}
A_i = \begin{bmatrix} A_i^{(1)} & & & \\
& A_i^{(2)} & & \\
& & \ddots & \\
& & & A_i^{(k)} \end{bmatrix},
\qquad i=1,\ldots,m,
\end{equation}
where each block $A^{(j)}_i$ is a $d\times d$ matrix with 
\begin{equation}\label{t3}
\tr(A^{(j)}_i)= \epsilon_j \qquad\text{for all } i=1,\ldots,m, \ j=1,\ldots,k.
\end{equation}
Then, for any $j$ we have
\begin{equation}
\label{t4}
\maxroot \mu[A_1^{(j)},\ldots, A_m^{(j)}]\leq \maxroot \mu[A_1,\ldots, A_m] -\sum_{l=1,\ l\ne j}^k \epsilon_l.
\end{equation}
\end{theorem}

\begin{remark}
In Theorem \ref{t1} it is not essential that all blocks of matrices $A_i$ have the same size $d$. Instead, one can assume that each block $A^{(j)}_i$ is a $d_j\times d_j$ matrix for some collection of dimensions $d_1,\ldots,d_k$. This more general result can be easily deduced from Theorem \ref{t1} by enlarging blocks $A^{(j)}_i$ to $d\times d$ matrices with $d=\max d_j$ and inserting zero rows and columns.
\end{remark}

\begin{proof}
If $k=1$, then the conclusion \eqref{t4} is trivial. Hence, we can assume that $k\ge 2$.
Enlarging each block $A^{(j)}_i$ by inserting zero rows and columns, we can assume that its dimension $d$ satisfies 
\begin{equation}\label{t6}
d \ge m+\sum_{i=1}^m \rank(A_i^{(1)}).
\end{equation}
By permuting, it suffices to prove \eqref{t4} for $j=k$. By \eqref{t6} we can find by induction positive semidefinite $d \times d$ matrices $Z_1,\ldots,Z_{m}$ of rank one such that
\begin{equation}\label{t8}
\begin{aligned}
\tr(Z_j)= \sum_{l=1}^{k-1} \epsilon_l &\qquad \text{for } j=1,\ldots,m,
\\
\range(Z_j) \perp \range(Z_i) \cup \range(A_i^{(1)})
&\qquad \text{for } j \ne i=1,\ldots,m.
\end{aligned}
\end{equation}
Consider a collection of $dk \times dk$ block diagonal matrices
$B_1,\ldots,B_m$ of the form
\[
B_i = \begin{bmatrix} Z_i & & & & \\
& \mathbf 0 & & & \\
& & \ddots & & \\
& & & \mathbf 0 & \\
& & & & A_i^{(k)} \end{bmatrix},
\qquad i=1,\ldots,m.
\]
Since
\[
A_1= \begin{bmatrix} A_1^{(1)} & & & \\
& \ddots & & \\
& & A_1^{(k-1)} & \\
& & & \mathbf 0 \end{bmatrix}
+
\begin{bmatrix} \mathbf 0 & & & \\
& \ddots & & \\
& & \mathbf 0 & \\
& & & A_1^{(k)} \end{bmatrix}
\]
by Lemma \ref{zl} and \eqref{t8} we have
\[
\maxroot \mu [A_1 ,A_2,\ldots,A_m]
\ge 
\maxroot \mu[B_1,A_2,\ldots,A_m].
\]
Thus, applying repeatedly Lemma \ref{zl} yields 
\begin{equation}\label{t10}
\begin{aligned}
\maxroot \mu[A_1,\ldots,A_m] \ge 
\maxroot \mu[B_1,A_2,\ldots,A_m] &\ge 
\ldots \\
& \ge \maxroot \mu[B_1,\ldots,B_m].
\end{aligned}
\end{equation}
For $i=1,\ldots,d$, let $E_i$ be $d \times d$  matrix with all zero entries except diagonal entry $(i,i)$, which is equal to $1$. By \eqref{t8} we can find $d \times d$ unitary matrix $U$ such that
\[
Z_i= \epsilon U E_i U^* \qquad\text{for all }i=1,\ldots,m, \quad\text{where }\epsilon:= \sum_{l=1}^{k-1} \epsilon_l.
\]
Define 
a collection of $dk \times dk$ block diagonal matrices
$C_1,\ldots,C_m$ of the form
\[
C_i = \begin{bmatrix} \epsilon E_i & & & & \\
& \mathbf 0 & & & \\
& & \ddots & & \\
& & & \mathbf 0 & \\
& & & & U^*A_i^{(k)} U \end{bmatrix}
\qquad i=1,\ldots,m.
\]
By \eqref{uau} we have $\mu[B_1,\ldots,B_m](z) = \mu[C_1,\ldots,C_m](z)$ for all $z\in\C$. Hence, Lemma \ref{shift} implies that
\[
\tilde \mu[B_1,\ldots,B_m](z) = \tilde\mu [U^*A_1^{(k)}U,\ldots,U^*A_m^{(k)}U](z-\epsilon) = \tilde\mu [A_1^{(k)},\ldots,A_m^{(k)}](z-\epsilon),
\qquad z\in \C.
\]
Combining this with \eqref{t10} yields
\[
\maxroot \mu[A_1,\ldots,A_m] \ge \maxroot \mu[B_1,\ldots,B_m] = \epsilon + \maxroot \mu[A^{(k)}_1,\ldots,A^{(k)}_m].
\]
This proves \eqref{t4}.
\end{proof}

We are now ready to complete the proof of Theorem \ref{thbx}.

\begin{proof}[Proof of Theorem \ref{thbx}]
First, we shall prove the estimate \eqref{thbx2} under the assumption that each block satisfies 
\begin{equation}\label{thb3}
\tr(X_i^{(j)}) = \epsilon_j\qquad\text{with probability $1$.}
\end{equation}
Let $\epsilon:= \sum_{l=1}^{k} \epsilon_l$.
Following the proof of Theorem \ref{thmp}, we apply Lemma \ref{max} and Theorem \ref{mixed}. That is, with positive probability we have
\begin{equation}\label{thb4}
\maxroot(\mu[X_1,\ldots X_m]) \le \maxroot(\mu[\E{}{X_1},\ldots,\E{}{X_m}]) \le (1+\sqrt{\epsilon})^2.
\end{equation}
However, instead of immediately using Lemma \ref{norm}, which yields 
\[
\bigg\| \sum_{i=1}^m X_i \bigg\| \le \maxroot(\mu[X_1,\ldots X_m]),
\]
we shall apply Theorem \ref{t1} instead. 

Choose an outcome for which \eqref{thb4} holds. This defines deterministic matrices $A_1,\ldots,A_m$ such that $X_i=A_i$, $i=1,\ldots,m$, for this outcome. By the hypothesis \eqref{thb1}, matrices $A_i$ are block diagonal of the form \eqref{t2}. By \eqref{thb3}, each block $A_i^{(j)}$ satisfies \eqref{t3}. Hence, by Lemma \ref{norm}, Theorem \ref{t1}, and \eqref{thb4} we have for $j=1,\ldots,k$,
\[
\begin{aligned}
\bigg\|\sum_{i=1}^m  A_i^{(j)} \bigg\| & \le  \maxroot \mu[A_1^{(j)},\ldots,A_m^{(j)}] 
 \le \maxroot \mu[A_1,\ldots,A_m] - \sum_{l=1,\ l\ne j}^k \epsilon_l
\\
& \le \bigg(1+\bigg(\sum_{l=1}^{k}\epsilon_l \bigg)^{1/2} \bigg)^2 - \sum_{l=1,\ l\ne j}^k \epsilon_l 
= 1 + 2\bigg(\sum_{l=1}^k \epsilon_l \bigg)^{1/2} + \epsilon_j.
\end{aligned}
\]

This proves \eqref{thbx2} under the additional hypothesis \eqref{thb3}.
To relax the assumption \eqref{thb3}, it suffices to find jointly independent $d'k\times d'k$ positive semidefinite random matrices $Y_1, \dots, Y_m$ in the block diagonal form \eqref{thb1}, which satisfy
\[
\sum_{i=1}^m \E{}{Y_i^{(j)}} \le \mathbf I_{d'},\qquad  \tr(Y_i^{(j)}) = \epsilon_j \quad\text{and}\quad X_i^{(j)} \oplus \mathbf 0_{d'-d} \le Y_i^{(j)}\text{ with probability 1}.
\]
This is possible for sufficiently large $d'>d$, which opens a room for increasing traces of each block $Y_i^{(j)}$ without violating the bounds in \eqref{thbx1}. In fact, we let $Y_i^{(j)}= X_i^{(j)} \oplus \xi^{(j)}_i\mathbf I_{d'-d}$ for suitably chosen random variables $0\le \xi^{(j)}_i\le 1$ satisfying
\[
\tr Y_i^{(j)} = \tr X_i^{(j)} + (d'-d) \xi^{(j)}_i = \epsilon_j \quad\text{ with probability 1}.
\]

Applying the special case of Theorem \ref{thbx} for $Y_1, \dots, Y_m$ yields the desired outcome such that
\[
\bigg\| \sum_{i=1}^m X_i^{(j)} \bigg\| 
\leq \bigg\| \sum_{i=1}^m Y_i^{(j)} \bigg\|
\leq 1 + 2\bigg(\sum_{l=1}^k \epsilon_l \bigg)^{1/2} + \epsilon_j
\qquad\text{for all }j \in [k].
\qedhere
\]
\end{proof}

\section{Selector form of Weaver's KS$_r$ conjecture}\label{S3}

In this section we show a selector form of Weaver's KS$_r$ conjecture for block diagonal trace class operators. This result has interesting consequences already in the case when operators are rank one. It implies a simultaneous selector result for multiple Bessel families, which can be used to deduce the multi-paving result of Ravichandran and Srivastava \cite{RS}. Further consequences include a selector variant of the Feichtinger conjecture for multiple Bessel families, which is explored in the next section.

We start by showing a selector variant of Weaver's KS$_r$ conjecture for trace class operators. The proof relies on Theorem \ref{thmp} following the reduction scheme as in \cite[Theorem 4.4]{AW}.

\begin{theorem}\label{ksr} Let $I$ be countable, $\epsilon>0$, and $r\ge 2$. Suppose that $\{T_i\}_{i\in I}$ is a family of positive trace class operators in a separable Hilbert space $\mathcal{H}$ satisfying 
\begin{equation}\label{ksr0}
\sum_{i\in I} T_i \le \bI \quad\text{and}\quad \tr(T_i)\le \epsilon \qquad\text{for all }i \in I.
\end{equation}
Let $\{ J_{k}\} _{k \in K}$ be a collection of disjoint subsets of $I$ with $\# J_k \geq r$,
	for all $k$. Then, there exists a selector $J\subset\bigcup_{k}J_{k}$
	satisfying 
	\begin{equation}\label{sel}
		\# J\cap J_{k} =1\qquad\forall k
	\end{equation}
such that 
\begin{equation}\label{ksr2}
\bigg\| \sum_{i\in  J} T_i  \bigg\| \le \bigg(\frac{1}{\sqrt{r}}+ \sqrt{\epsilon} \bigg)^2 .
\end{equation}
\end{theorem}

The series in \eqref{ksr0} is assumed to converge in the strong operator topology, but not necessarily in operator norm.

\begin{proof}
First, we consider the case when the index set $I$ is finite. For any $i\in I$, choose a sequence of positive finite rank operators $\{T^{(n)}_i\}$ such that
\begin{equation}\label{iaw3}
0 \le T_i^{(n)} \le T_i \qquad\text{and}\qquad \lim_{n\to\infty} || T_i - T^{(n)}_i || =0.
\end{equation}
Take any $n\in \N$. Since operators $\{T^{(n)}_i\}_{i\in I}$ act nontrivially on some finite dimensional subspace $\mathcal K \subset \mathcal H$ we can identify them with positive semidefinite matrices on $\C^d$, $d =\dim \mathcal K$.

Without loss of generality we can assume that $\#J_{k}= r$ for all $k$. For fixed $n\in \N$, define independent positive semidefinite random matrices $X_k$, $k\in K$, such that each $X_k$ takes $r$ values equal to $r T^{(n)}_i$ for $i\in J_k$ with probability $\frac{1}{r}$.
Note that
\[
	\sum_{k\in K} \E{}{X_k} = \sum_{k\in K} \sum_{i \in J_k}  T^{(n)}_i \le \sum_{i\in I} T_i \le \mathbf I.
\]
\[
\E{}{\tr (X_k)} = \sum_{i \in J_k} \tr(T^{(n)}_i) \le \sum_{i \in J_k} \tr(T _i)
\leq r\epsilon.
\]
By Theorem \ref{thmp} there exists an outcome such that
\[
\bigg\| \sum_{k\in K} X_k \bigg\| \le (1+ \sqrt{\epsilon r})^2.
\]
This implies the existence of a selector set $J\subset I$ depending on $n\in \N$ and satisfying \eqref{sel} such that

\begin{equation}\label{ksr5}
	\bigg\| \sum_{i\in J}T_j^{(n)} \bigg\| \leq \bigg(\frac{1}{\sqrt{r}}+\sqrt{\epsilon}\bigg)^{2}.
\end{equation}
Since $I$ is finite, by the pigeonhole principle there exists a single selector set $J$ satisfying the above bound for infinitely many $n\in \N$. Letting $n\to \infty$, \eqref{iaw3} yields \eqref{ksr2}.

Finally, suppose that $I$ is infinite. By reindexing we can assume that $K=\N$. For every $n\in \N$, applying the above conclusion to a finite family $T_i$, $i \in \bigcup_{k\in [n]} J_k$, yields a selector $J=J(n)$ of the family $ \{J_k: k\in [n]\}$ such that \eqref{ksr2} holds.
 By \cite[Lemma 3.7]{BL}, which is a combination of diagonal argument with the pigeonhole principle, there exists a selector $J_\infty \subset I$ and an increasing sequence $\{n_j\}$ such that
\[
J(n_j) \cap \bigg( \bigcup_{k=1}^j J_k \bigg) = J_\infty \cap \bigg( \bigcup_{k=1}^j J_k \bigg) \qquad\text{for all }j.
\]
Since for any $j$ we have
\begin{equation}\label{ksr6}
\bigg\| \sum_{i\in  J(n_j)} T_i  \bigg\| \le \bigg(\frac{1}{\sqrt{r}}+ \sqrt{\epsilon} \bigg)^2,
\end{equation}
we deduce that \eqref{ksr2} holds with $J=J_\infty$.
\end{proof}

In the case $r=2$ we can show a stronger variant of the selector KS$_r$ result, which also controls the operator norms from below. For this we need to require that the family $\{J_k\}_{k\in K}$ forms a partition of $I$.

\begin{theorem}\label{ks2} Let $I$ be countable and $\epsilon>0$. Suppose that $\{T_i\}_{i\in I}$ is a family of positive trace class operators in a separable Hilbert space $\mathcal{H}$ satisfying 
\begin{equation}\label{ks20}
T:= \sum_{i\in I} T_i \le \bI \quad\text{and}\quad \tr(T_i)\le \epsilon \qquad\text{for all }i \in I.
\end{equation}
Let $\{ J_{k}\} _{k \in K}$ be a partition of $I$ with $\# J_k =2 $,
	for all $k$. Then, there exists a selector $J\subset\bigcup_{k}J_{k}$
	satisfying \eqref{sel}
such that 
\begin{equation}\label{ks22}
\bigg\| \sum_{i\in  J} T_i - \frac12 T \bigg\|, 
\bigg\| \sum_{i\in  I \setminus J} T_i - \frac12 T \bigg\|
 \le 2  \sqrt{\epsilon} + \epsilon.
\end{equation}
\end{theorem}

\begin{proof}
First, we assume that $I$ is finite and operators $T_i$ have finite rank. For any $k\in K$ we write $J_k=\{i_k, i'_k\}$. Define independent positive semidefinite random matrices $X_k$, $k\in K$, such that each $X_k$ is block diagonal taking two values  
\[
X_k = \begin{bmatrix} X_{k}^{(1)} & \\ & X_k^{(2)}
\end{bmatrix}
=
2\begin{bmatrix} T_{i_k} & \\ & T_{i'_k}
\end{bmatrix},
\ 
2\begin{bmatrix} T_{i'_k} & \\ & T_{i_k}
\end{bmatrix}
\qquad\text{each with probability $\frac{1}{2}$.}
\]
Note that
\[
	\sum_{k\in K} \E{}{X_k} = \sum_{k\in K} \sum_{i \in J_k}  \begin{bmatrix} T_{i} & \\ & T_{i}
\end{bmatrix} 
= \begin{bmatrix} T & \\ & T \end{bmatrix}
\le \begin{bmatrix} \mathbf I & \\ & \mathbf I
\end{bmatrix}.
\]
\[
\tr X^{(1)}_{k}, \tr X^{(2)}_{k} \le 2\ve \qquad\text{for all }k\in K.
\]
Choose a collection of deterministic positive semidefinite random matrices $X'_k$, $k\in K'$ of the form
\[
X_k'= \begin{bmatrix} T'_k  & \\ & T'_k  \end{bmatrix}
\]
such that $\tr T'_k \le 2\epsilon$ and $\sum_{k\in K'} T'_k = \mathbf I - T$.

By Theorem \ref{thb} there exists an outcome such that
\[
\bigg\| \sum_{k\in K} X_k + \sum_{k\in K'} X'_k  \bigg\| \le 1+ 4\sqrt{\epsilon}+2\epsilon.
\]
This implies the existence of a selector set $J\subset I$ depending on $n\in \N$ and satisfying \eqref{sel} such that
\[
\begin{bmatrix} 2\sum_{i\in J}T_i + (\mathbf I -T) & \\
&  2\sum_{i\in I \setminus J}T_i + (\mathbf I - T) 
\end{bmatrix} \le  (1 + 4  \sqrt{\epsilon} + 2\epsilon) \begin{bmatrix} \mathbf I & \\ & \mathbf I
\end{bmatrix}.
\]
Hence, we have
\begin{equation}\label{ks25}
	 \sum_{i\in J}T_i ,  \sum_{i\in I \setminus J}T_i 
	\leq  \frac12 T +  (2 \sqrt{\epsilon} + \epsilon)\mathbf I.
\end{equation}

Now, using an approximation argument as in the proof of Theorem \ref{ksr}, we can show that \eqref{ks25} holds for trace class operators as well. Likewise, using the diagonal argument we can relax the assumption that the index set $I$ is finite. Consequently, we deduce the existence of selector $J \subset I$ such that \eqref{ks25} holds for a family of trace class operators $\{T_i\}_{i\in I}$ satisfying  \eqref{ks20}.

Finally, we use the assumption that $\sum_{i\in I} T_i = T$ and \eqref{ks25} to deduce the lower bound
\[
\sum_{i\in J}T_i  = T -  \sum_{i\in I \setminus J}T_i 
	\geq  \frac12 T - (2 \sqrt{\epsilon} + \epsilon) \mathbf I.
\]
Likewise,
\[
\sum_{i\in I \setminus J}T_i 
	\geq  \frac12 T -  (2 \sqrt{\epsilon} + \epsilon) \mathbf I.
\qedhere 
\]
\end{proof}

Using Theorem \ref{thbx} instead of Theorem \ref{thmp} we obtain an extension of Theorem \ref{ksr} for block diagonal operators.

\begin{theorem}\label{ksb} Let $I$ be countable and $r\ge 2$. Let $N$ be at most countable and let $\epsilon_j>0$ be such that $\sum_{j\in N} \epsilon_j<\infty$. Suppose that $\{T_i= \bigoplus_{j\in N} T^{(j)}_i  \}_{i\in I}$ is a family of positive trace class operators in a separable Hilbert space $\mathcal{H}= \bigoplus_{j\in N} \mathcal H_j$ satisfying 
\begin{equation}\label{ksb0}
\sum_{i\in I} T_i \le \bI \quad\text{and}\quad \tr(T^{(j)} _i)\le \epsilon_j \qquad\text{for all }i \in I, j\in N.
\end{equation}
Let $\{ J_{k}\} _{k \in K}$ be a collection of disjoint subsets of $I$ with $\# J_k \geq r$,
	for all $k$. Then, there exists a selector $J\subset\bigcup_{k}J_{k}$
	satisfying \eqref{sel}
such that for all $j\in N$ we have
\begin{equation}\label{ksb2}
\bigg\| \sum_{i\in  J} T^{(j)}_i  \bigg\| \le \frac 1r + \epsilon_j + 2\bigg(\sum_{l \in N} \frac{\epsilon_l}r \bigg)^{1/2} .
\end{equation}
\end{theorem}

\begin{proof} We follow closely the proof of Theorem \ref{ksr} with the exception that we apply Theorem \ref{thbx} instead of Theorem \ref{thmp}. First, we consider the case when the index sets $I$ and $N$ are finite and operators $T_i$ are finite rank.  By reindexing we can assume $N=[m]$ for $m\in \N$. Then, we consider independent positive semidefinite random matrices $X_k$, $k\in K$, such that each $X_k$ takes $r$ values equal to $r T_i$ for $i\in J_k$ with probability $\frac{1}{r}$. Each random matrix $X_k$ is block diagonal with blocks $X_k^{(j)}$, $j\in [m]$, satisfying
\[
\tr X_k^{(j)} = r \tr T^{(j)}_i \le r\epsilon_j \qquad\text{for some }i\in J_k.
\]
Hence, by Theorem \ref{thbx}, there exists a selector set $J$ such that for all $j\in [m]$ we have
\begin{equation}\label{ksb4}
\bigg\| \sum_{i\in J} r T_i^{(j)} \bigg\| \leq 1 + 2\bigg(\sum_{l=1}^m  r\epsilon_l \bigg)^{1/2} + r\epsilon_j.
\end{equation}
Dividing by $r$ yields \eqref{ksb2}.

Next, we assume that $T_i$ are trace class operators and both $I$ and $N$ are still finite. Approximating $T_i$ by finite rank operators as in Theorem \ref{ksr} yields  \eqref{ksb2} by the pigeonhole principle. In a similar way we can relax the assumption that $N$ is finite. Indeed, assume $N$ is infinite. By reindexing we can assume $N=\N$. For every $m\in \N$, we consider truncated trace class operators
$\bigoplus_{j\in [m]} T^{(j)}_i$, $i\in I$, acting on $\bigoplus_{j\in [m]} \mathcal H_j$. Hence, there exists a selector $J\subset I$ depending on $m\in N$ and satisfying \eqref{sel} such that \eqref{ksb2} holds for all $j\in [m]$. Since $I$ is finite, by the pigeonhole principle there exists a single selector set $J$ satisfying the above bound for infinitely many $m\in \N$. Consequently, \eqref{ksb2} holds for all $j\in \N$.

Finally, we relax the assumption that $I$ is finite in the same way as in the proof of Theorem \ref{ksr}. That is, we find selectors $J(n)$, $n\in \N$, for increasingly larger index sets $I$. Then using \cite[Lemma 3.7]{BL} we deduce the existence of a selector $J=J_\infty$ such that 
\eqref{ksb2} holds for all $j\in N$.
\end{proof}

Given a vector $u\in \mathcal H$, define a rank one positive operator $u\otimes u: \mathcal H \to \mathcal H$ by
\[
(u \otimes u)(v) = \langle v, u \rangle u \qquad\text{for } v\in\mathcal H.
\]
Theorem \ref{ksb} has interesting consequences already in the case when operators $T^{(j)}_i$ have all rank $1$. Letting $T^{(j)}_i = u^{(j)}_i \otimes u^{(j)}_i$ for some vectors $u^{(j)}_i \in \mathcal H_j$, we obtain the following corollary about simultaneous selector for multiple Bessel sequences. Corollary \ref{ksc} is a selector variant of Weaver's KS$_r$ conjecture \cite{We} for a (possibly infinite) collection of Bessel sequences, which was shown by Marcus, Spielman, and Srivastava \cite[Corollary 1.5]{MSS}

\begin{corollary}\label{ksc} Let $I$ be countable and $r\ge 2$. Let $N$ be at most countable and let $\epsilon_j>0$ be such that $\sum_{j\in N} \epsilon_j<\infty$. Suppose that $\{u^{(j)}_i\}_{i\in I}$ is a Bessel sequence  with bound $1$, in a Hilbert space $\mathcal H_j$, $j\in N$, such that
\begin{equation}\label{ksc0}
\|u^{(j)}_i\|^2 \le \epsilon_j 
\qquad\text{for all }i \in I, j\in N.
\end{equation}
Let $\{ J_{k}\} _{k \in K}$ be a collection of disjoint subsets of $I$ with $\# J_k \geq r$,
	for all $k$. Then, there exists a selector $J\subset\bigcup_{k}J_{k}$
	satisfying \eqref{sel}
such that for all $j\in N$, $\{u^{(j)}_i\}_{i\in J}$ is a Bessel sequence with bound at most
\begin{equation}\label{ksc2}
 \frac 1r + \epsilon_j + 2\bigg(\sum_{l \in N} \frac{\epsilon_l}r \bigg)^{1/2} .
\end{equation}
\end{corollary}

Every selector result considered in this paper has a corresponding result about partitions. We shall illustrate this principle by deducing a partition variant of Theorem \ref{ks2}. Corollary \ref{pks2} is also an immediate consequence of \cite[Theorem 4.1]{AW}, which yields a slightly better bound $2  \sqrt{\epsilon}$ in \eqref{pks22}. 

\begin{corollary}\label{pks2} Let $I$ be countable and $\epsilon>0$. Suppose that $\{T_i\}_{i\in I}$ is a family of positive trace class operators in a separable Hilbert space $\mathcal{H}$ satisfying 
\begin{equation}\label{pks21}
T:= \sum_{i\in I} T_i \le \bI \quad\text{and}\quad \tr(T_i)\le \epsilon \qquad\text{for all }i \in I.
\end{equation}
Then, there is a partition $\{S_1,S_2\}$ of $I$ such that
\begin{equation}\label{pks22}
\bigg\| \sum_{i\in S_k} T_i - \frac12 T \bigg\|
 \le 2  \sqrt{\epsilon} + \epsilon \qquad\text{for } k=1,2.
\end{equation}
\end{corollary}

\begin{proof}
Define block diagonal operators $\tilde T_{i,k}$, $(i,k) \in I \times [2]$ on $\mathcal H \oplus \mathcal H$ by
\[
\tilde T_{i,1} = \begin{bmatrix} T_i & \\ & 0 \end{bmatrix},
\qquad
\tilde T_{i,2} = \begin{bmatrix} 0 & \\ & T_i \end{bmatrix}.
\]
Let $J_i = \{ (i,1), (i,2) \}$, $i\in I$, be a dumb partition of $I \times [2]$. The family $\{\tilde T_{i,k}\}_{(i,k) \in I \times [2]}$ satisfies the assumption of Theorem \ref{ks2} with the sum
\[
\tilde T =  \begin{bmatrix} T & \\ & T \end{bmatrix}.
\]
Hence, there exists a selector $J$ such that 
\[
\bigg\| \sum_{(i,k) \in J} \tilde T_{i,k} - \frac12 \tilde T  \bigg\|  \le 2  \sqrt{\epsilon} + \epsilon.
\]
Letting $S_1=\{i \in I: (i,1) \in J \}$ and $S_2 = \{i \in I: (i,2) \in J \}$ yields \eqref{pks22}. 
\end{proof}

As a further illustration of how selector results imply partition results, we show that a special case of Corollary \ref{ksc}, when $\epsilon_j=1/2$, can be used to deduce the multi-paving result of Ravichandran and Srivastava \cite{RS}.

\begin{definition}
Let $T$ be a bounded operator on $\ell^2(I)$. Let $r\in \N$ and $\ve>0$. We say that $T$ has $(r,\ve)$-paving if there exists partition $\{S_1,\ldots,S_r\}$ of $I$ such that
\begin{equation}\label{mp0}
||P_{S_k} T P_{S_k}|| \le \ve ||T|| \qquad\text{for }k\in [r].
\end{equation}
Here, for $S \subset I$, we let $P_S$ denote a diagonal projection of $\ell^2(I)$ onto $\overline{\spa}\{e_j: j\in S\}$. We say that a finite family of operators $T_1,\ldots,T_m$ has $(r,\ve)$-multi-paving if there exists a single partition $\{S_1,\ldots,S_r\}$ of $I$ such that \eqref{mp0} holds for every $T=T_j$, where $j\in [m]$.
\end{definition}

\begin{corollary}\label{ksd} Let $I$ be countable and $m\ge 2$. Suppose that $\{u^{(j)}_i\}_{i\in I}$ is a Bessel sequence  with bound $1$, in a Hilbert space $\mathcal H_j$, $j\in [2m]$, such that
\begin{equation}\label{ksd0}
\|u^{(j)}_i\|^2 =\frac 12 
\qquad\text{for all }i \in I.
\end{equation}
Let $\ve>0$. Then, for any $r \ge 18m/\ve^2$, there exists a partition $\{S_1, \ldots, S_r \}$ of $I$, such that each subfamily $\{u^{(j)}_i\}_{i\in S_k}$, $k\in [r]$, is Bessel with bound at most $\frac{1+\ve}2$. In other words, there exists an $(r,\frac{1+\ve}2)$-multi-paving of the operators represented by the Gram matrices $(\langle u^{(j)}_i, u^{(j)}_{i'}\rangle)_{i,i' \in I}$ for all $j\in [2m]$.
\end{corollary}

\begin{proof}
The paving conclusion is nontrivial only if $\frac{1+\ve}2<1$.
Hence, without loss of generality we can assume that $0<\ve<1$.
For $j\in [2m]$, let $\{v^{(j)}_{i,k}\}_{(i,k)\in I\times [r]}$ be a multiple copy of a Bessel sequence $\{u^{(j)}_i\}_{i\in I}$, which lives in the Hilbert space $(\mathcal H_j)^{\oplus r}$, given by
\[
\begin{aligned}
v^{(j)}_{i,1} &= (u^{(j)}_{i},0, \ldots, 0), \\
v^{(j)}_{i,2} &= (0,u^{(j)}_{i},0, \ldots, 0), \\
& \vdots\\
v^{(j)}_{i,r} &= (0, \ldots, 0, u^{(j)}_{i}).
\end{aligned}
\]
Let $\{J_i\}_{i\in I}$ be a dumb partition of $I \times [r]$ given by $J_i=\{i\} \times [r]$. By Corollary \ref{ksc}, there exists a selector $J \subset \bigcup_i J_i$ satisfying 
\[
\#(J \cap J_i) =1 \qquad\forall i\in I,
\]
and such that $\{v^{(j)}_{i,k}\}_{(i,k)\in J}$ is a Bessel sequence with bound at most
\[
\frac12 + \frac 1r + 2\bigg(\sum_{l=1}^{2m} \frac{1}{2r} \bigg)^{1/2} = \frac12 + \frac 1r +2\sqrt{m/r}.
\]
Since $r \ge 18m/\ve^2$ and $m\ge 2$, the above bound is dominated by
\[
\frac12 + \frac{\ve^2}{18m} + \frac{2\ve}{\sqrt{18}} \le \frac12 + \ve(1/36+2/\sqrt{18}) < \frac{1+\ve}2.
\]
Now, the selector set $J$ corresponds to a partition $\{S_1, \ldots, S_r \}$ of $I$ as follows. An element $i\in I$ belongs to $S_k$ for some $k\in [r]$ if and only if $(i,k) \in J$. Consequently, for all $j\in [2m]$ and $k\in [r]$,  a family $\{u^{(j)}_i\}_{i\in S_k}$ is Bessel with bound $\frac{1+\ve}2$. In other words, a family of the Gram matrices $(\langle u^{(j)}_i, u^{(j)}_{i'}\rangle)_{i,i' \in I}$, $j\in [2m]$, has an $(r,\frac{1+\ve}2)$-multi-paving.
\end{proof}

A standard reduction argument reduces the problem of paving of operators with zero diagonal to the  multi-paving of pairs of projections with constant diagonal equal $\frac12$, see \cite[Section 2.2]{MB}. Hence, we recover the main result of Ravichandran and Srivastava \cite[Theorem 1.1]{RS} on multi-paving of hermitian matrices.

\begin{corollary}\label{pavrs}
Let $\ve>0$ and $m\ge 2$.
A finite collection $T_1,\ldots,T_m$ of self-adjoint operators on $\ell^2(I)$ with zero diagonals has $(r,\ve)$-multi-paving provided that $r \ge 18m/\ve^2$. That is, there exists a partition $\{S_1,\ldots,S_r\}$ of $I$ such that for all $j\in [m]$,
\[
||P_{S_k} T_j P_{S_k}|| \le \ve ||T_j|| \qquad\text{for }k\in [r].
\]
\end{corollary}

\begin{proof}
By \cite[Lemma 2.4]{MB}, the $(r,\ve)$-paving property for a self-adjoint operator $T$ with zero diagonal on $\ell^2(I)$ follows from the $(r,\ve)$-paving of a reflection operator $R$ on $\ell^2(I)\oplus \ell^2(I)$, which satisfies $R=R^*$ and $R^2=\mathbf I$. By the standard argument as in \cite[Lemma 2.5]{MB}, the $(r,\ve)$-paving of a reflection $R$ with zero diagonal follows from the $(r,\tfrac{1+\ve}2)$-multi-paving of the pair of projections $(\mathbf I \pm R)/2$, which have constant diagonals equal $\frac12$. It is well-known, see \cite[Lemma 2.6]{MB}, that any projection $P$ on $\ell^2(I)$ with constant diagonal equal $\frac12$ is represented by the Gram matrix of a Parseval frame $\{u_i\}_{i\in I}$ in $\ell^2(I)$ with equal norms $||u_i||^2=\frac12$, $i\in I$. 

An analogous statement hold for a family $T_1,\ldots,T_m$ of self-adjoint operators on $\ell^2(I)$ with zero diagonal. That is, $(r,\ve)$-multi-paving of 
$T_1,\ldots,T_m$ can be reduced to a $(r,\tfrac{1+\ve}2)$-multi-paving of a family of projections which are represented by the Gram matrices corresponding to a collection of $2m$ Parseval frames $\{u^{(j)}_i\}_i$ in $\ell^2(I) \oplus \ell^2(I)$ of equal norms $||u^{(j)}_i||^2=\frac12$, where $j\in [2m]$. Therefore, Corollary \ref{ksd} yields the required conclusion.
\end{proof}

\section{Selector form of multi Feichtinger's conjecture}\label{S4}

The author and Londner \cite[Theorem 2.1]{BL} have shown a selector variant of Feichtinger's conjecture. In this section we show an extension of this result for multiple Bessel sequences. We also show a generalization of the $R_\epsilon$ conjecture of Casazza, Tremain, and Vershynin, see \cite[Conjecture 3.1]{CT}, in the form of a selector theorem. 

In general, selector results are stronger than their partition counterparts. In particular, a selector form of Feichtinger's conjecture implies the usual partition form of Feichtinger's conjecture. We shall illustrate this by showing Feichtinger's conjecture for (possibly) infinite collection of Bessel sequences. 

\begin{theorem}\label{BL} 
There exist universal constants $c, C>0$ such that the following holds.
Let sets $I$ and $N$ be at most countable. Let $0<\epsilon_j<1$ be such that $\sum_{j\in N} (1-\epsilon_j) <\infty$. Let $j_0 \in N$ be such that the value $\epsilon_{j}$ is the smallest. Define
\begin{equation}\label{BLf}
r:=  C\bigg(1+ \frac{1}{(\epsilon_{j_0})^2} \sum_{j\in N:\ \epsilon_j < 1/2} \epsilon_j \bigg)  \bigg( \sum_{ j\in N:\ \epsilon_j < 1/2} \frac{\epsilon_{j_0}}{\epsilon_j} + \sum_{j\in N:\ \epsilon_j \ge 1/2} (1 -\epsilon_j) \bigg).
\end{equation}
Suppose that $\{u^{(j)}_i\}_{i\in I}$ is a Bessel sequence  with bound $1$, in a Hilbert space $\mathcal H_j$, $j\in N$, such that
\begin{equation}\label{BLe}
\|u^{(j)}_i\|^2 \ge \epsilon_j 
\qquad\text{for all }i \in I, j\in N.
\end{equation}
Let $\{ J_{k}\} _{k \in K}$ be a collection of disjoint subsets of $I$ with 
\begin{equation}\label{BLa}
\# J_k \geq r 
\qquad\text{for all } k \in K
\end{equation}
Then, there exists a selector $J\subset\bigcup_{k}J_{k}$
	satisfying \eqref{sel}
such that for all $j\in N$, $\{u^{(j)}_i\}_{i\in J}$ is a Riesz sequence
	in $\mathcal H_j$ with lower bound $c \epsilon_j$. In addition, if we have equality in \eqref{BLe}, then the upper Riesz bound of $\{u^{(j)}_i\}_{i\in J}$ is $2 \epsilon_j$ for all $j\in N$. 
\end{theorem}

Before proving Theorem \ref{BL} we show how it implies the corresponding partition result.

\begin{corollary} \label{BLP}
Under the assumptions on multiple Bessel sequences  $\{u^{(j)}_i\}_{i\in I}$, $j\in N$, as in Theorem \ref{BL}, there exists a partition $\{S_1,\ldots,S_r\}$ of $I$ such that for all $j\in N$ and $k=1,\ldots,r$, $\{u^{(j)}_i\}_{i\in S_k}$ is a Riesz sequence in $\mathcal H_j$ with lower bound $c \epsilon_j$. In addition, if we have equality in \eqref{BLe}, then the upper Riesz bound of $\{u^{(j)}_i\}_{i\in S_k}$ is $2 \epsilon_j$ for all $j\in N$. 
\end{corollary}

\begin{proof} We mimic the proof Corollary \ref{ksd} by considering multiple copies $\{v^{(j)}_{i,k}\}_{(i,k)\in I\times [r]}$ of a Bessel sequence $\{u^{(j)}_i\}_{i\in I}$, living in $r$ different compartments of the Hilbert space $(\mathcal H_j)^{\oplus r}$. Let $\{J_i\}_{i\in I}$ be a dumb partition of $I \times [r]$ given by $J_i=\{i\} \times [r]$. Applying Theorem \ref{BL} yields a selector $J \subset \bigcup_i J_i$ such that $\{v^{(j)}_{i,k}\}_{(i,k)\in J}$ is a Riesz sequence. A selector $J$ corresponds to a partition $\{S_1, \ldots, S_r \}$ of $I$ given by
\[
S_k=\{ i\in I: (i,k) \in J\}, \qquad k\in [r].
\]
Then, we conclude that Theorem \ref{BL} yields Corollary \ref{BLP}.
\end{proof}

To prove Theorem \ref{BL} we follow the scheme employed in the proof of \cite[Theorem 2.1]{BL}, which is an adaptation of the proof of the asymptotic bounds in the Feichtinger conjecture \cite[Theorem 6.11]{BCMS}. In both cases, the key role is played by a result on Naimark's complements \cite[Proposition 5.4]{BCMS} and a result on completing a Bessel sequence with bound $1$ to a Parseval frame. For convenience we combine these two facts into one result.

\begin{lemma}\label{nc}
	Let $\mathcal H$ be a Hilbert space and $0<\epsilon <1$. Suppose that
	$\{ u_i \} _{i\in I }\subset\mathcal H$ is a Bessel sequence
	with bound $1$ and $||u_i||^{2}\geq \epsilon$ for all $i$. Then, there exists a Hilbert space $\mathcal H'$ and a Bessel sequence $\{ v_i \} _{i\in I }\subset\mathcal H'$ with bound $1$ and $||v_i||^{2} = 1- ||u_i||^{2} \leq 1-\epsilon$ for all $i$, such that  for any $J \subset I$ and $\delta>0$, the following two statements are equivalent:
	\begin{enumerate}[(i)]
	\item $\{ u_i \} _{i \in J }$ is a Riesz sequence with lower Riesz bound $\delta$,
	\item $\{v_i\}_{i\in J}$ is a Bessel sequence with bound $1-\delta$.
	\end{enumerate}
\end{lemma}
	
\begin{proof} Suppose first $\{ u_i \} _{i\in I }$ is a Parseval frame in $\mathcal H$. By Naimark's dilation theorem, the space $\mathcal H$ can be embedded into $\ell^2(I)$ and $u_i=P e_i$ for all $i$, where $P$ is an orthogonal projection of $\ell^2(I)$ onto $\mathcal H$ and $\{e_i\}_{i\in I} $ is a standard basis of $\ell^2(I)$. The Naimark complement is defined as $v_i=(\mathbf I - P) e_i$, $i\in I$. Then, $\{ v_i \} _{i \in I}$ is a Parseval frame for the orthogonal complement $\mathcal H^\perp = \ell^2(I) \ominus \mathcal H$. By \cite[Proposition 5.4]{BCMS}, statements (i) and (ii) are equivalent.

Suppose next that $\{ u_i \} _{i\in I }\subset\mathcal H$ is merely a Bessel sequence with bound $1$ and $||u_i||^{2}\geq \epsilon$ for all $i$. By adding a collection of extra vectors $\{\vp_i\}_{i\in \N} \subset \mathcal H$, we can make the system $\{ u_i \} _{i\in I } \cup \{\vp_i\}_{i\in \N} $ a Parseval frame in $\mathcal H$. Indeed, consider the frame operator of $\{ u_i \} _{i\in I } $ given by
\[
S= \sum_{i\in I} u_i \otimes u_i.
\]
The operator $S$ is positive semidefinite and $ \mathbf 0 \le S \le \mathbf I$. If the operator $\mathbf I - S$ is compact, then it can be represented in terms of its eigenvalue sequence $1\ge \lambda_1 \ge \lambda_2 \ge \ldots$ and the corresponding orthonormal basis of eigenvectors $\{w_i\}_{i\in \N}$ as 
\[
\mathbf I - S = \sum_{i=1}^\infty \lambda_i w_i \otimes w_i.
\]
In the case when $\mathbf I - S$ is finite rank the above sum has finitely many nonzero terms. Hence, the vectors $\vp_i= \sqrt{\lambda_i} w_i$, $i\in \N$, do the trick. Suppose next that the operator $\mathbf I - S$ is not compact. Let $c>0$ be a constant smaller than the essential norm of $\mathbf I -S$. By the result of Dykema, Freeman, Kornelson, Larson, Ordower, and Weber \cite[Theorem 2]{dfklow}, this operator can be represented as a scalar multiple of orthogonal projections, see also \cite{amrs}. Hence, there exists a sequence of unit vectors $\{w_i\}_{i\in \N}$ such that
\[
\mathbf I - S = \sum_{i=1}^\infty c \langle \cdot , w_i \rangle w_i.
\]
Hence, the vectors $\vp_i= \sqrt{c} w_i$, $i\in \N$, work.
Let $\{ v_i \} _{i\in I } \cup \{\tilde \vp_i\}_{i\in \N} $ be the Naimark complement of the Parseval frame $\{ u_i \} _{i \in I } \cup \{\vp_i\}_{i\in \N}$. Clearly, $\{ v_i \} _{i\in I} $ is a Bessel sequence with bound $1$ and $||v_i||^2=1-||u_i||^2 \le 1-\epsilon$. The previously shown case yields the required equivalence.
\end{proof}	

Combining Lemma \ref{nc} with Corollary \ref{ksc} yields a preliminary variant of Theorem \ref{BL} that is optimal for nearly unit vectors. 

\begin{theorem}\label{BL2}
 Let sets $I$ and $N$ be at most countable. Let $0<\epsilon_j<1$ be such that 
 \begin{equation}\label{BL2a}
\delta_0:= \sum_{j\in N} (1-\epsilon_j) <\infty.
 \end{equation}
 Let $j_0 \in N$ be such that the value $\epsilon_{j}$ is the smallest. Define
 \begin{equation}\label{BL2b}
 r:=\begin{cases} 2 & \text{if } \delta_0 < 3/2 - \sqrt{2} \approx 0.0857864, \\
\lceil 21 \delta_0/(\epsilon_{j_0})^2 \rceil & \text{otherwise.}
\end{cases}
\end{equation}
 Suppose that $\{u^{(j)}_i\}_{i\in I}$ is a Bessel sequence  with bound $1$, in a Hilbert space $\mathcal H_j$, $j\in N$, such that
\begin{equation*}
\|u^{(j)}_i\|^2 \ge \epsilon_j 
\qquad\text{for all }i \in I, j\in N.
\end{equation*}
Let $\{ J_{k}\} _{k \in K}$ be a collection of disjoint subsets of $I$ with $\# J_k =r $ for all $k$. Then, there exists a selector $J\subset\bigcup_{k}J_{k}$
	satisfying \eqref{sel}
such that for all $j\in N$, $\{u^{(j)}_i\}_{i\in J}$ is a Riesz sequence
	in $\mathcal H_j$ with lower Riesz bound $\ge c \epsilon_j$, where $c$ is a universal constant. 
\end{theorem}

\begin{proof}
By Lemma \ref{nc} for each $j\in N$, there exists a Bessel sequence $\{ v_i^{(j)} \} _{i\in I }$ with bound $1$ satisfying $||v^{(j)}_i||^{2}\leq 1-\epsilon_j$ for all $i\in I$, such that an improved Bessel bound on a subsequence $\{ v_i^{(j)} \} _{i\in J }$, where $J\subset I$, implies a lower Riesz bound on $\{ u_i^{(j)} \} _{i\in J }$, and vice versa.

Suppose that $\delta_0<   3/2 - \sqrt{2}$ and hence $r=2$. Applying Corollary \ref{ksc} to Bessel sequences $\{ v_i^{(j)} \} _{i\in I }$ yields a selector $J \subset I$ such that for every $j\in N$, a subsequence $\{ v_i^{(j)} \} _{i\in J }$ has Bessel bound at most
\[
(1-\epsilon_j)  + \frac 12 + \sqrt{2\delta_0} \le \frac12 + \delta_0 + \sqrt{2\delta_0}  .
\]
By Lemma \ref{nc} each subsequence $\{u^{(j)}_i\}_{i\in J}$ has lower Riesz bound $\ge 1/2- \delta_0 - \sqrt{2\delta_0}>0$.

Suppose next that $\delta_0 \ge   3/2 - \sqrt{2}$ and hence $r$ is given by \eqref{BL2b}. Applying Corollary \ref{ksc} to Bessel sequences $\{ v_i^{(j)} \} _{i\in I }$ yields a selector $J \subset I$ such that for every $j\in N$, a subsequence $\{ v_i^{(j)} \} _{i\in J }$ has Bessel bound at most
\[
 (1-\epsilon_j) +\frac 1r + \frac{2\sqrt{\delta_0} }{\sqrt r} 
 \le (1-\epsilon_{j}) + \frac{(\epsilon_{j_0})^2}{ 21 \delta_0} + \frac{2\epsilon_{j_0}}{\sqrt{21}}  \le 
 (1-\epsilon_{j}) + \epsilon_{j_0}\bigg( \frac{1}{21\delta_0} + \frac2{\sqrt{21}} \bigg) .
\]
By Lemma \ref{nc} each subsequence $\{u^{(j)}_i\}_{i\in J}$ has lower Riesz bound $\ge \epsilon_j - \epsilon_{j_0} ( \frac{1}{21\delta_0} + \frac2{\sqrt{21}} ) \ge c \epsilon_j$ since $\frac{1}{21\delta_0} + \frac2{\sqrt{21}} <1$.
\end{proof}

Theorem \ref{BL2} is far from optimal when the quantity $\delta_0$ is large. Indeed, in the case of a single Bessel sequence, by \cite[Theorem 2.1]{BL} the optimal size $r$ is bounded by $O(1/\ve_0)$ as $\ve_0 \to 0$ instead of $O(1/(\ve_0)^2)$ as implied by \eqref{BL2b}. This necessitates a more elaborate proof of Theorem \ref{BL}, which brings small $\epsilon_j$'s above $1/2$, leading to a more complicated formula \eqref{BLf}

\begin{proof}[Proof of Theorem \ref{BL}]
Without loss of generality, we can assume that
\begin{equation}\label{BLi}
\|u^{(j)}_i\|^2 = \epsilon_j 
\qquad\text{for all }i \in I, j\in N.
\end{equation}
Indeed, replacing vectors $u^{(j)}_i$ by $(\epsilon_j/||u^{(j)}_i||) u^{(j)}_i$ decreases both the Bessel bound of $\{u^{(j)}_i\}_{i\in I}$ and the lower Riesz bound of $\{u^{(j)}_i\}_{i\in J}$, where $J\subset I$. Hence, if we show the conclusion for vectors satisfying \eqref{BLi}, then we immediately deduce the lower Riesz bound conclusion for vectors as in \eqref{BLe}. However, for the upper Riesz bound conclusion to hold, we must necessarily assume a stronger assumption \eqref{BLi}.

Let $c$ be the constant as in Theorem \ref{BL2}. Let $C$ be a sufficiently large constant to be determined later. 
We claim that there exists a redundant selector $\tilde J \subset \bigcup_k J_k$, satisfying for all $k \in K$,
\begin{equation}\label{BL5}
\#(\tilde J\cap J_k) \ge \tilde r:= \frac{C}{6} \bigg( \sum_{ j\in N:\ \epsilon_j < 1/2} \frac{\epsilon_{j_0}}{\epsilon_j} + \sum_{j\in N:\ \epsilon_j \ge 1/2} (1 -\epsilon_j) \bigg),
\end{equation}
and such that for every $j\in N$, the system $\{u^{(j)}_i\}_{i\in \tilde J}$ is a Bessel sequence with bound $\le 2\ve_j$. This is a trivial statement unless $\epsilon_{j_0}<1/2$. 
 In this case we divide each set $J_k$ into $\lceil \tilde r\rceil $ subsets of cardinality at least 
\[
r'= \frac{6}{(\epsilon_{j_0})^2} \sum_{j\in N'} \epsilon_j, 
\qquad\text{where }  N':=\{j\in N: \epsilon_j <1/2\}.
\]
By Corollary \ref{ksc}, there exists a selector $\tilde J$ of the resulting family such that for every $j\in N'$, the system $\{u^{(j)}_i\}_{i\in \tilde J}$ is a Bessel sequence with bound
\[
\begin{aligned}
B_j:=\frac 1{r'} + \epsilon_j + 2 \bigg(\sum_{l \in  N'} \frac{\epsilon_l}{r'} \bigg)^{1/2} 
&\le \epsilon_j \bigg( 1+ \frac 1{\epsilon_{j_0} r'}  + \frac{2}{\epsilon_{j_0}} \bigg(\sum_{l \in N'} \frac{\epsilon_l}{r'} \bigg)^{1/2} \bigg)
\\
&\le \epsilon_j \bigg(1+ \frac16+ \frac{2}{\sqrt{6}} \bigg) <2\epsilon_j.
\end{aligned}
\]
In the penultimate inequality we used a trivial bound $r' \ge 6/\epsilon_{j_0}$. Since $\tilde J$ is a selector of a family obtained by dividing each set $J_k$ into $\lceil \tilde r\rceil $ subsets, the estimate \eqref{BL5} follows. 

After renormalizing the family $\{\tilde u^{(j)}_i:= (B_j)^{-1/2} u^{(j)}_i\}_{i\in \tilde J}$ is a Bessel sequence with bound $1$ and
\[
\| \tilde u^{(j)}_i \|^2 \ge \tilde \epsilon_j:= \frac{\epsilon_j}{B_j} \ge 1- \frac 1{r' \epsilon_j } - \frac{2}{\epsilon_j} 
\bigg(\sum_{l \in N'} \frac{\epsilon_l}{r'} \bigg)^{1/2},
\]
since $1/(1+x) \ge 1-x$ for $x\ge 0$. Observe that
\[
\begin{aligned}
\sum_{j \in N'} (1- \tilde \epsilon_j) 
\le 
  \sum_{j \in N'}  \frac 1{\epsilon_j }\bigg( \frac 1{r'} + 2
\bigg(\sum_{l \in N'} \frac{\epsilon_l}{r'} \bigg)^{1/2} \bigg)
\le 
 \sum_{j \in N'}  \frac 1{\epsilon_j } \bigg( \frac {\epsilon_{j_0}}{6} + \frac{2\epsilon_{j_0}}{\sqrt 6} \bigg) 
 \le 2  \sum_{j \in N'}  \frac{\epsilon_{j_0} } {\epsilon_j }.
\end{aligned}
\]

Next we apply Theorem \ref{BL2} to the conglomerate of Bessel sequences $\{u^{(j)}_i\}_{i\in \tilde J}$, $j\in N \setminus N'$, and Bessel sequences $\{\tilde u^{(j)}_i\}_{i\in 
\tilde J}$, $j\in N'$.  The corresponding quantity $\delta_0$ in \eqref{BL2a} satisfies
\[
\delta_0 = \sum_{j \in N'} (1- \tilde \epsilon_j) + \sum_{j\in N \setminus N'} (1-\epsilon_j) 
\le  2  \sum_{j \in N'}  \frac{\epsilon_{j_0} } {\epsilon_j }+  \sum_{j\in N \setminus N'} (1-\epsilon_j) .
\]
Hence, by choosing sufficiently large $C>0$, \eqref{BL5} implies that
\[
\#(\tilde J\cap J_k) \ge \lceil 21 \delta_0/(\tilde \epsilon_{j_0})^2 \rceil  \qquad\text{for all }k\in K,
\]
where $\tilde \epsilon_{j_0}$ is the smallest among $\tilde \epsilon_j$, $j\in N'$ and $\epsilon_j$, $j\in N \setminus N'$. In particular, $\tilde \epsilon_{j_0} \ge 1/2$.
By Theorem \ref{BL2} there exists a selector $J \subset\tilde J$ such that each system $\{u^{(j)}_i\}_{i\in J}$, $j\in N \setminus N'$ is a Riesz sequence with lower frame bound $c\epsilon_j$. Also each system $\{\tilde u^{(j)}_i\}_{i\in J}$, $j\in  N'$ is a Riesz sequence with lower frame bound  $c \tilde \epsilon_j$. Thus,  $\{u^{(j)}_i\}_{i\in J}$ is a Riesz sequence with lower frame bound  $c \epsilon_j$ for $j\in  N'$.
Finally, the upper Riesz bound of $\{u^{(j)}_i\}_{i\in J}$, which is equal to the Bessel bound, is at most $2\epsilon_j$.
\end{proof}

\begin{remark} The assumption in Theorem \ref{BL} that each system $\{u^{(j)}_i\}_{i\in I}$ is a Bessel sequence with bound $1$ is made merely for simplicity. Instead, we can merely assume that for every $j\in N$, there exist constants $B_j,\eta_j>0$ such that the system $\{u^{(j)}_i\}_{i\in I}$ is a Bessel sequence with bound $B_j$, and 
\begin{equation*}
\|u^{(j)}_i\|^2 \ge \eta_j 
\qquad\text{for all }i \in I, j\in N.
\end{equation*}
The important quantity here is the ratio $\epsilon_j=\eta_j/B_j$. If the sequence $\{\epsilon_j\}_{j\in N}$ fulfills numerical requirements of Theorem \ref{BL}, then under hypothesis \eqref{BLa} we can conclude the existence of selector $J$ such that $\{u^{(j)}_i\}_{i\in J}$ is a Riesz sequence with lower Riesz bound $\ge c \eta_j$. Indeed, this follows by applying Theorem \ref{BL} to normalized Bessel sequences $\{(B_j)^{-1/2}u^{(j)}_i\}_{i\in I}$.
\end{remark}

In the case when all vectors $\{u^{(j)}_i\}_{i\in I}$, which form a Bessel sequence, have equal norms, we can show a generalization of the $R_\epsilon$ conjecture of Casazza, Tremain, and Vershynin, see \cite[Conjecture 3.1]{CT}. The following result is an improvement of \cite[Theorem 3.8]{BL},  which established a selector result of size $r=\lceil C\frac{B}{\ve^{4}}\rceil$ for a single Bessel sequence with bound $B$, in two aspects. Not only does Theorem \ref{reps} apply to possibly infinite collections of Bessel sequences, but at the same time it yields an improved size $r=\lceil C\frac{B}{\ve^{2}}\rceil$ for a single Bessel sequence. When specialized to exponential systems, this is also an improvement of \cite[Theorem 6.4]{CT2}.

\begin{theorem}\label{reps}
There exists a universal constant $C>0$ such that the following holds.
Let sets $I$ and $N$ be at most countable.  Let $0<\ve<1$. Suppose that $\{u^{(j)}_i\}_{i\in I}$ is a Bessel sequence  with bound $B_j$, in a Hilbert space $\mathcal H_j$, $j\in N$, such that
\begin{equation*}
\|u^{(j)}_i\|^2 = 1
\qquad\text{for all }i \in I, j\in N.
\end{equation*}
Define $\epsilon_j=1/B_j$ and suppose that 
$\sum_{j\in N} (1-\epsilon_j) <\infty$. Let $j_0 \in N$ be such that the value $\epsilon_{j}$ is the smallest. Define
\[
r:=  \frac{C}{\ve^2}\bigg(1+ \frac{1}{(\epsilon_{j_0})^2} \sum_{j\in N: \epsilon_j < 1/2} \epsilon_j \bigg)  \max\bigg( 1,\sum_{ j\in N: \epsilon_j < 1/(1+\ve)}\epsilon_j + \sum_{j\in N} (1 -\epsilon_j) \bigg).
\]

Assume $\{ J_{k}\} _{k \in K}$ is a collection of disjoint subsets of $I$ satisfying
\[
\# J_k \geq r
\qquad\text{for all } k \in K.
\]
Then, there exists a selector $J\subset\bigcup_{k}J_{k}$
	satisfying \eqref{sel}
such that each system $\{u^{(j)}_i\}_{i\in J}$ is a Riesz sequence
	in $\mathcal H_j$ with bounds $1-\ve$ and $1+\ve$ for all $j\in N$.
\end{theorem}

In the proof of Theorem \ref{reps} we shall employ \cite[Lemma 6.13]{BCMS}.

\begin{lemma}\label{dri}
	Suppose $\{u_{i}\}_{i\in I}$ is a Riesz basis in $\mathcal{H}$, and let $\{u^{*}_{i}\}_{i\in I}$ 
	be its unique biorthogonal Riesz basis. Then for any subset $J\subset I$, the Riesz sequence
	bounds of $\{u_{i}\}_{i\in J}$ are A and B if and only if the Riesz sequence bounds of
	$\{u^{*}_{i}\}_{i\in J}$ are $1/B$ and $1/A$. 
\end{lemma}

\begin{proof}[Proof of Theorem \ref{reps}] The proof consists of two steps. In the first step we preprocess Bessel sequences $\{u^{(j)}_i\}_{i\in I}$ to obtain Bessel sequences $\{u^{(j)}_i\}_{i\in \tilde J}$ with bounds $\le 2$ for appropriate choice of $\tilde J\subset I$. This is done simultaneously for all $j\in N$. In the second step we find a selector $J \subset \tilde J$ such that the resulting systems $\{u^{(j)}_i\}_{i\in J}$ are Riesz sequences with bounds $1\pm \ve$ for all $j\in N$.

{\bf Step 1.}
Let $C$ be a sufficiently large constant to be determined later. 
We claim that there exists a redundant selector $\tilde J \subset \bigcup_k J_k$, satisfying for all $k \in K$,
\begin{equation}\label{reps3}
\#(\tilde J\cap J_k) \ge \tilde r:= \frac{C}{6\ve^2} \max \bigg(1,\sum_{ j\in N: \epsilon_j < 1/(1+\ve)}\epsilon_j + \sum_{j\in N} (1 -\epsilon_j)  \bigg),
\end{equation}
and such that for every $j\in N$, the system $\{u^{(j)}_i\}_{i\in \tilde J}$ is a Bessel sequence with bound $2$. This is a trivial statement unless there exists $j\in N$ such that $B_j>2$. 
 In this case we divide each set $J_k$ into $\lceil \tilde r\rceil $ subsets of cardinality at least 
\[
r'= \frac{6}{(\epsilon_{j_0})^2} \sum_{j\in N'} \epsilon_j, 
\qquad\text{where }  N':=\{j\in N: \epsilon_j <1/2\}.
\]
By Corollary \ref{ksc}, there exists a selector $\tilde J$ of the resulting family such that for every $j\in N'$, the system $\{u^{(j)}_i\}_{i\in \tilde J}$ is a Bessel sequence with bound at most
\[
\frac{1}{\epsilon_j} \bigg(\frac 1{r'} + \epsilon_j + 2 \bigg(\sum_{l \in  N'} \frac{\epsilon_l}{r'} \bigg)^{1/2} \bigg)
=
 1+ \frac 1{\epsilon_j r'}  + \frac{2}{\epsilon_j} \bigg(\sum_{l \in N'} \frac{\epsilon_l}{r'} \bigg)^{1/2} .
\]
Since $\epsilon_{j_0}$ is the smallest value among $\epsilon_j$, $ j \in N'$, the above expression is bounded by
\[
 1+ \frac 1{\epsilon_{j_0} r'}  + \frac{2}{\epsilon_{j_0}} \bigg(\sum_{l \in N'} \frac{\epsilon_l}{r'} \bigg)^{1/2} \le 1+ \frac16+ \frac{2}{\sqrt{6}} <2.
\]
In the second inequality we used a trivial bound $r' \ge 6/\epsilon_{j_0}$. Since $\tilde J$ is a selector of a family obtained by dividing each set $J_k$ into $\lceil \tilde r\rceil $ subsets, the estimate \eqref{reps3} follows. 

{\bf Step 2.} Define
\[
\epsilon'_j = \begin{cases} 1/2 & j \in N',
\\
\epsilon_j & j \in N \setminus N'.
\end{cases}
\]
By the above construction, for each $j\in N$, the system $\{\sqrt{\epsilon'_j} u^{(j)}_i\}_{i\in \tilde J}$ is a Bessel sequence with bound $1$.
Applying Lemma \ref{nc} to a Bessel sequence $\{\sqrt{\epsilon'_j} u^{(j)}_i\}_{i\in \tilde J}$ yields the corresponding Bessel sequence $\{\sqrt{1-\epsilon'_j} v^{(j)}_i\}_{i\in \tilde J}$ with bound $1$. Since the vectors $u^{(j)}_i$ have unit norm, so do the vectors $v^{(j)}_i$.
Define
\[
\tilde N := \{j\in N: \epsilon'_j< 1/(1+\ve)\}.
\]
Next, we apply Corollary \ref{ksc} to the conglomerate of Bessel sequences $\{\sqrt{\epsilon'_j} u^{(j)}_i\}_{i\in \tilde J}$ for $j\in \tilde N$ and $\{\sqrt{1-\epsilon'_j} v^{(j)}_i\}_{i\in \tilde J}$ for $j\in N$. 
The reason for this choice is that we need to control the upper Bessel bound only when $j\in \tilde N$, whereas the lower Riesz bound needs to be controlled for all $j\in N$.
This yields  a selector $J\subset\bigcup_{k}J_{k}$
such that:
\begin{enumerate}[(i)]
\item for all $j\in \tilde N$, $\{\sqrt{\epsilon'_j}  u^{(j)}_i\}_{i\in J}$  is a Bessel sequence with bound at most
\begin{equation}\label{reps7}
\epsilon'_j+ \frac{1}{\tilde r}+  \frac{2}{ \sqrt{\tilde r}} \bigg(\sum_{l \in \tilde N} \epsilon'_l + \sum_{l \in N} (1-\epsilon'_l )\bigg)^{1/2}
\le  \epsilon'_j+ \frac{6 \ve^2}C+ \frac{2\sqrt 6 \ve}{\sqrt C} .
\end{equation}
\item
for all $j\in N$,  $\{\sqrt{1-\epsilon'_j} v^{(j)}_i\}_{i\in  J}$ is a Bessel sequence with bound at most
\begin{equation}\label{reps8}
(1-\epsilon'_j)+ \frac{1}{\tilde r}+  \frac{2}{ \sqrt{\tilde r}} \bigg(\sum_{l \in \tilde N} \epsilon'_l + \sum_{l \in N} (1-\epsilon'_l )\bigg)^{1/2}
\le (1-\epsilon'_j)+ \frac{6 \ve^2}C+ \frac{2\sqrt 6 \ve}{\sqrt C}.
\end{equation}
\end{enumerate}
Indeed, the above inequalities follow by the definition of $\tilde r$ in \eqref{reps3},
\[
\begin{aligned}
 \tilde r &=  \frac{C}{6\ve^2} \max\bigg( 1,\#\{ j\in N: \epsilon_j < 1/(1+\ve)\} + \sum_{j\in N: \epsilon_j \ge 1/(1+\ve) } (1 -\epsilon_j) \bigg)
\\
& = \frac{C}{6\ve^2} \max\bigg(1, \sum_{l \in \tilde N} \epsilon'_l + \sum_{l \in N} (1-\epsilon'_l ) \bigg).
 \end{aligned}
\]
Recall that $\epsilon_j' \ge 1/2$ for all $j\in N$. By choosing sufficiently large $C>0$, say $C>12 (3 + 2 \sqrt{2})$, (i)  implies that the Bessel bound of $\{ \sqrt{\epsilon'_j} u^{(j)}_i\}_{i\in J}$ is at most $\epsilon'_j(1+\ve)$ for all $j\in \tilde N$. If $j\in N \setminus \tilde N$, then $\epsilon_j (1+\ve) \ge 1$, while the Bessel bound is at most $1$. Thus, the Bessel bound of $\{ \sqrt{\epsilon'_j} u^{(j)}_i\}_{i\in J}$ is at most $\epsilon'_j(1+\ve)$ for all $j\in N$. Likewise, (ii) implies that the Bessel bound of  $\{\sqrt{1-\epsilon'_j} v^{(j)}_i\}_{i\in  J}$ is at most $(1-\epsilon'_j)+\ve\epsilon'_j$. 
By Lemma \ref{nc}, $\{ \sqrt{\epsilon'_j} u^{(j)}_i\}_{i\in J}$ is a Riesz sequence with lower bound 
$\epsilon'_j(1-\ve)$ for all $j\in N$. Consequently, for all $j\in N$, the system $\{ u^{(j)}_i\}_{i\in J}$ is a Riesz sequence with bounds $1-\ve$ and $1+\ve$.
\end{proof}

It is not clear whether the bound in Theorem \ref{reps} is asymptotically optimal as a function of the collection of parameters $\{\epsilon_j\}_{j\in N}$. For simplicity, assume that all $\epsilon_j$ are equal, which by the Blaschke condition $\sum_{j\in N} (1-\epsilon_j) <\infty$, necessarily forces the set $N$ to be finite. Corollary \ref{mfei} is an extension of \cite[Theorem 3.8]{BL} in two ways. It yields a simultaneous selector for multiple Bessel sequences. At the same time it gives an improved, asymptotically optimal bound for a single Bessel family in terms of Riesz sequence tightness parameter $\ve>0$, which was conjectured in \cite[Theorem 3.8]{BL}.

\begin{corollary}\label{mfei}
There exists a universal constant $C>0$ such that the following holds.
Let $m\in \N$, $B>1$, and $0<\ve<1$.  Suppose that for each $j\in [m]$, $\{u^{(j)}_i\}_{i\in I}$ is a Bessel sequence with bound $B$ and
\[
\|u^{(j)}_i\|^2 = 1
\qquad\text{for all }i \in I, j\in [m].
\]
Let $\{ J_{k}\} _{k \in K}$ be a collection of disjoint subsets of $I$ with 
\begin{equation}\label{mfei2}
\# J_k \geq C\frac{m^2 B}{\ve^2}
\qquad\text{for all } k \in K.
\end{equation}
Then, there exists a selector $J\subset\bigcup_{k}J_{k}$
	satisfying \eqref{sel}
such that for all $j\in [m]$, $\{u^{(j)}_i\}_{i\in J}$ is a Riesz sequence
with bounds $ 1-\ve$ and $1+\ve$. 
\end{corollary}

Note that the size of sets $J_k$ in Corollary \ref{mfei} is proportional to the Bessel bound $B$, and inversely proportional to $\ve^2$, which controls the tightness of resulting Riesz sequence. These bounds are asymptotically optimal as either $B \to \infty$ or $\ve \to 0$, see \cite{BCMS, RS}. However, it is not clear at all whether the size of sets $J_k$ needs to grow proportionally to $m^2$. The paving result of Ravichandran and Srivastava, Corollary \ref{pavrs}, suggests that a linear growth might be possible.

\begin{problem}
Is the bound in \eqref{mfei2} asymptotically optimal bound as $m\to \infty$?
\end{problem}

\section{Iterated KS$_2$ result for binary selectors}\label{S5}

In this section we  show an iterative selector form of Theorem \ref{ks2}. A precursor of this result can be traced to the works of Nitzan, Olevskii, and Ulanovskii \cite{NOU} and Freeman and Speegle \cite{FS}, which exploited an iteration of the KS$_2$ result, see also \cite[Section 10.4]{OU}. Our result has three main advantages over earlier results. 

1) (Frame vs. Bessel) Theorem \ref{ks2i} yields a sparsification of the frame operator not only for frames, but also for Bessel sequences. Consequently, it yields a more efficient proof of the discretization problem bypassing complicated estimates present in the original argument in \cite{FS}. At the same time, it gives nearly tight discretization of continuous Parseval frames, which was an open problem unresolved by techniques in \cite{FS}.

2) (Rank one vs. higher rank) At the same time our result applies to a more general setting of trace class operators, which implies discretization results for positive (trace class) operator-valued Bessel families. Continuous frames are merely a special class of such families corresponding to rank one operators. This will be shown in Section \ref{S7}.

3) (Partition vs. selector) Finally, our result takes a form of a selector result, which gives an extra control on the choice of sparsification. This is in contrast to original sparsification results, where no control on resulting partitions was present. As an application we deduce a variant of Feichtinger's conjecture in the case when a Parseval frame in Theorem \ref{BL} is nearly unit norm. That is, the quantity $\sum_{j\in N} (1-\epsilon_j)$ is very small. We show the existence of a sparse selector for nearly unit norm Parseval frame, which upon its removal yields a Riesz sequence. This will be shown in Section \ref{S6}, thus solving an open problem on Parseval frames of exponentials \cite[Open Problem 2]{BL}.

To formulate this result it is convenient to define the concept of iterative binary selectors.

\begin{definition}\label{binse}
Let $I$ be countable. Let $\{ J_{k}\} _{k \in K}$ be any partition of $I$ with $\# J_k =2$ for all $k$. Binary selectors of order $1$ are sets $I_0$ and $I_1$ such that $I=I_0 \cup I_1$ and
\[
\#(I_0 \cap J_k)=\#(I_1 \cap J_k) = 1
\qquad\text{for all } k\in K.
\]
For $N\ge 2$, we define selectors of order $N$ inductively. Suppose that binary selectors $I_b$,  $b\in \{0,1\}^{N-1}$, of order $N-1$ are already defined. For given $b\in \{0,1\}^{N-1}$, let $\{ J_{k}\} _{k \in K}$ be any partition of $I_b$ with $\# J_k =2$ for all $k$. Binary selectors of order $N$ are sets $I_{b0}$ and $I_{b1}$ satisfying $I_b=I_{b0} \cup I_{b1}$ and
\[
\#(I_{b0} \cap J_k)=\#(I_{b1} \cap J_k) = 1
\qquad\text{for all } k\in K.
\]
\end{definition}

We also need an elementary numerical lemma.

\begin{lemma}\label{numer} There exists an absolute constant $C>1$ such that the following holds. Let $\delta>0$. Define sequence $\{B_j\}_{j=0}^\infty$ recursively by
\[
B_0=1, \qquad B_{j+1}=B_j + 4 \sqrt{ 2^j \delta B_j } + 2^{j+1} \delta, \ j \ge 0.
\]
Then, for any $N\in \N$ such that $2^N < 1/\delta$, we have
\[
B_N \le 1+C \sqrt{ 2^{N}\delta}.
\]
\end{lemma}

\begin{proof}
Note that $B_j \ge 1$ for all $j\ge 0$. Hence, for any $j\ge 0$ such that $2^{j}\delta <1$ we have
\[
B_{j+1} \le 1 + 4 B_j \sqrt{2^j \delta } + 2^{j+1} \delta \le 
B_j(1 + 4 \sqrt{2^j \delta } + 2\sqrt{ 2^{j} \delta}) 
=
B_j(1 + 6 \sqrt{2^j \delta }).
\]
Thus, for any $k\in \N$, 
\[
B_k \le \prod_{j=0}^{k-1} (1 + 6 \sqrt{2^j \delta })\le \exp \bigg(6 \sqrt{\delta} \sum_{j=0}^{k-1} 2^{j/2} \bigg)
\le\exp(6(\sqrt{2}+1)  \sqrt{2^{k}\delta}).
\]
There exists a constant $c>0$ such that
$1+x \le e^x \le 1+cx$ for $x\in [0,6(\sqrt{2}+1)]$. Hence, 
\[
B_N \le \exp(6(\sqrt{2}+1)  \sqrt{2^{N}\delta})
\le  1+ 6 c (\sqrt{2}+1) \sqrt{2^N\delta}.
\]
Choosing $C= 6 c (\sqrt{2}+1)$ yields the required estimate.
\end{proof}

\begin{theorem}\label{ks2i} Let $I$ be countable and $\delta>0$. Suppose that $\{T_i\}_{i\in I}$ is a family of positive trace class operators in a separable Hilbert space $\mathcal{H}$ satisfying 
\begin{equation}\label{ks2i2}
T:= \sum_{i\in I} T_i \le \bI \quad\text{and}\quad \tr(T_i)\le \delta \qquad\text{for all }i \in I.
\end{equation}
Let $N\in \N$ be such that $2^N < 1/\delta$. For any intermediate choices of partitions with sets of size 2, there exist binary selectors $I_b$,  $b\in \{0,1\}^{N}$, that form a partition of $I$, and
\begin{equation}\label{ks2i1}
\bigg\| 2^N \sum_{i\in  I_b} T_i - T \bigg\| \le  C \sqrt{2^N \delta} \qquad\text{for all }b\in \{0,1\}^{N},
\end{equation}
where $C>1$ is an absolute constant.
\end{theorem}

\begin{remark}\label{game}
Before proceeding with the proof we explain a formal meaning of ``intermediate choices" in Theorem \ref{ks2i}. It is most convenient to do this in the language of game theory. Suppose that Player A chooses any partition of the index set $I$ with sets of size $2$. Then, Player B chooses a binary selector for this partition consisting of sets $I_0$ and $I_1$. At the stage $j\in \N$ of the game we have binary selectors $I_b$, $b\in \{0,1\}^j$, of order $j$. For any $b\in \{0,1\}^j$, Player A chooses a partition of $I_b$ with sets of size $2$. Then, Player B chooses a binary selector $I_b$, $b\in \{0,1\}^{j+1}$, of order $j+1$. The game ends after $N$ steps. Player B wins if all binary selectors of order $N$ satisfy the bound \eqref{ks2i1}. Otherwise, Player A wins. Theorem \ref{ks2i} states that there is a winning strategy for Player B regardless of intermediate choices of partitions with sets of size 2 made by Player A.
\end{remark}

\begin{proof}
Let  $\{B_j\}_{j=0}^\infty$ be the sequence defined in Lemma \ref{numer}.
Let $\{ J_{k}\} _{k \in K}$ be any partition of $I$ with $\# J_k =2$ for all $k$. By Theorem \ref{ks2}, there exists selectors $I_0$ and $I_1$ of order $1$ such that
\[
\bigg\| \sum_{i\in  I_0} 2T_i -  T \bigg\|, 
\bigg\| \sum_{i\in  I_1} 2T_i -  T \bigg\|
 \le 4  \sqrt{\delta} + 2\delta=B_1-1.
\]
In particular, for $b\in \{0,1\}$ we have
\[
\sum_{i\in  I_b} 2T_i \le B_1\mathbf I.
\]
Suppose that for some $1 \le j <N$, selectors of order $j$ of the form $I_{b'0}$ and $I_{b'1}$ are already defined and satisfy for all $b'\in \{0,1\}^{j-1}$, 
\begin{align*}
& \sum_{i\in  I_{b'}} 2^{j-1} T_i \le B_{j-1} \mathbf I,
 \\
\bigg\| \sum_{i\in  I_{b'0}} 2^{j}T_i -   \sum_{i\in  I_{b'}} 2^{j-1} T_i \bigg\|, 
& \bigg\| \sum_{i\in  I_{b'1}} 2^{j}T_i -   \sum_{i\in  I_{b'}} 2^{j-1} T_i \bigg\|   \le  B_{j}-B_{j-1}.
\end{align*}
Here, we use the convention that $I_{b'}=I$ when $j=1$.
In particular, for any $b\in \{0,1\}^{j}$ we have 
\[
\sum_{i\in  I_b} 2^{j}T_i \le (B_{j}-B_{j-1}) \mathbf I + B_{j-1}\mathbf I = B_j \mathbf I.
\]
For given $b\in \{0,1\}^{j}$, let $\{ J_{k}\} _{k \in K}$ be any partition of $I_b$ with $\# J_k =2$ for all $k$. Observe that the family $\{(2^{j}/B_{j}) T_i \}_{i\in I_b}$ fulfills the assumptions of Theorem \ref{ks2} with $\ve =2^j\delta/B_j$. Hence, there exist selectors $I_{b0}$ and $I_{b1}$ of order $j+1$ such that
\[
\bigg\| \sum_{i\in  I_{b0}} 2^{j+1}T_i -   \sum_{i\in  I_b} 2^j T_i \bigg\|, 
\bigg\| \sum_{i\in  I_{b1}} 2^{j+1}T_i -   \sum_{i\in  I_b} 2^j T_i \bigg\|
 \le 4  \sqrt{2^j\delta B_j} + 2^{j+1}\delta=B_{j+1}-B_j.
\]
Take any $b\in \{0,1\}^{N}$. For any $ 0 \le j \le N$, let $b_j$ be the first $j$ components of $b$.  By telescoping we have
\[
\bigg\| 2^N \sum_{i\in  I_b} T_i - T \bigg\| \le \sum_{j=0}^{N-1} \bigg\| \sum_{i\in  I_{b_{j+1}}} 2^{j+1}T_i -   \sum_{i\in  I_{b_j}} 2^j T_i \bigg\|
\le \sum_{j=0}^{N-1} (B_{j+1}-B_j) = B_N-1.
\]
Then, Lemma \ref{numer} yields the estimate \eqref{ks2i1}.
\end{proof}

\begin{remark}\label{binary}
We remark that Theorem \ref{ks2i} also holds when the index set $I$ is finite after the following adaptation. If the cardinality of $I$ is divisible by $2^N$, then binary selectors of order $N$ exist and are well-defined. Otherwise, we enlarge the index set $I' \supset I$, so that the cardinality of $I'$ is divisible by $2^N$, and declare that $T_i$ is a zero operator for $i\in I' \setminus I$. 
Appending the family $\{T_i\}_{i\in I}$ by (phantom) zero operators does not affect the definition of the operator $T=\sum_{i\in I'} T_i$, nor the conclusion \eqref{ks2i1}. 

An alternative approach involves modifying the definition of binary selectors as follows.
For $j\ge 1$, we define selectors of order $j$ inductively. Suppose that binary selectors $I_b$,  $b\in \{0,1\}^{j-1}$, of order $j-1$ are already defined. In the base case $j=1$, we declare that $I_b=I$.
For given $b\in \{0,1\}^{j-1}$, let $\{ J_{k}\} _{k \in K}$ be any partition of $I_b$ with $\# J_k =1,2$ for all $k$. Binary selectors of order $j$ are disjoint sets $I_{b0}$ and $I_{b1}$ satisfying $I_b=I_{b0} \cup I_{b1}$ and
\[
\#(I_{b0} \cap J_k), \#(I_{b1} \cap J_k) \le 1
\qquad\text{for all } k\in K.
\]
That is, if $\# J_k=2$, then two elements of $J_k$ are assigned to each set $I_{b0}$ and $I_{b_1}$ as in Definition \ref{binse}. However, if $\# J_k=1$, then the unique element of $J_k$ is assigned to either $I_{b0}$ or $I_{b1}$.
\end{remark}

\section{Feichtinger's conjecture for nearly unit norm Parseval frames} \label{S6}

As an application of Theorem \ref{ks2i} we show a selector form of Feichtinger's conjecture for nearly unit norm Parseval frames. As a consequence of this result we solve a problem on Parseval frames of exponentials, which was  posed by Londner and the author in \cite[Open Problem 2]{BL}.

Let $X$ be a discrete metric space with distance $d$. Suppose that the counting measure on $X$ is doubling.  That is, there exists a constant $C>0$ such that
\begin{equation}\label{doubling}
\# B(x,2r) \le C\# B(x,r) \qquad\text{for all }x\in X, r>0.
\end{equation}
Here, $B(x,r)$ denotes that ball with center $x\in X$ and radius $r>0$. Observe that doubling condition automatically implies that $X$ is at most countable. The main example we have in mind is a full rank lattice $\Z^d$ with the usual Euclidean distance.

For the following lemma, it is convenient to implicitly assume that $X$ is infinite. However, a similar result holds if $X$ is finite by adding extra elements and making the cardinality of $X$ divisible by $2^N$.

\begin{lemma}\label{metr}
There exists a constant $\eta\in \N$, which depends only on a doubling constant of a discrete metric space $(X,d)$, such that the following happens. Suppose $N\in \N$ and $r>0$ are such that
\begin{equation}\label{cep}
\sup_{x\in X} \# B(x,r) \le 2^{N-\eta}.
\end{equation}
Then, there exists a choice of consecutive partitions of $X$ into non-empty sets of size $\le 2$, such that every binary selector $I_b$,  $b\in \{0,1\}^{N}$, of order $N$, is uniformly discrete. That is,
\begin{equation}\label{sepa}
\inf \{ d(x,y): x, y \in I_b, \ x\ne y \} \ge r \qquad\text{for all } b \in \{0,1\}^{N}.
\end{equation}
\end{lemma}

\begin{remark}
Before proceeding with the proof we explain a formal meaning of ``consecutive partitions" in Lemma \ref{metr}. We will use the language of game theory. The game proceeds in a similar way as in Remark \ref{game}. Player A chooses a partition of the index set $I$ with non-empty sets of size $\le 2$, whereas Player B chooses a binary selector for this partition consisting of sets $I_0$ and $I_1$. At the stage $j\in \N$ of the game we have binary selectors $I_b$, $b\in \{0,1\}^j$, of order $j$. Here, we allow modified binary selectors as in Remark \ref{binary}.
For any $b\in \{0,1\}^j$, Player A chooses a consecutive partition of $I_b$ with non-empty sets of size $\le 2$, whereas Player B chooses a binary selector $I_b$, $b\in \{0,1\}^{j+1}$. The game ends after $N$ steps. Player A wins if all binary selectors of order $N$ are uniformly discrete \eqref{sepa}. Otherwise, Player B wins. Lemma \ref{metr} states that there is a winning strategy for Player A regardless of the choice of binary selectors made by Player B.
\end{remark}

\begin{proof}
Let $\eta\in \N$ be sufficiently large constant to be chosen later. We choose the maximal set $X' \subset X$ such that the balls $B(x,r)$, $x\in X'$, are pairwise disjoint. This implies that the balls $B(x,2r)$, $x\in X'$, cover the space $X$. Let $K_x$, $x\in X'$, be a partition of $X$ such that
\[
B(x,r) \subset K_x \subset B(x,2r) \qquad\text{for all }x\in X'.
\]
Observe that $\# K_x \le C 2^{N-\eta}$ for all $x\in X'$. 

For every $j=1,\ldots,N$, we shall choose a partition $\{J^{(j)}_k\}_{k\in \N}$ of $X$ into sets of size $\le 2$, in two stages. First, we choose a partition $\{J^{(1)}_k\}_{k\in \N}$ of $X$ into sets of size $2$ such that for every $k\in \N$, $J^{(1)}_k \subset K_x$ for some $x\in X'$. This is certainly possible whenever $\# K_x$ is even. If some set $K_x$ has an odd number of elements, then we are left with one orphan element, say $y_0$, that does not have a match in $K_x$. In this case one of the sets $J^{(1)}_k$ is a singleton of $y_0$. Let $I_0$ and $I_1$ be any binary selectors corresponding to this partition. That is each element of $J^{(1)}_k$ is assigned to either $I_0$ or $I_1$. If $J^{(1)}_k=\{y_0\}$ is a singleton, then only one of the sets $I_0$ or $I_1$ gets a contribution of $y_0$.

Suppose that for $j\in \N$ such that $2^j< C2^{N-\eta}$, we have constructed a partition $\{J^{(j)}_k\}_{k\in \N}$ of $X$ into non-empty sets of size $\le 2$. Let $I_b$, $b\in \{0,1\}^{j}$, be binary selectors corresponding to this partition. Next, we choose  a partition $\{J^{(j+1)}_k\}_{k\in \N}$
 of $X$ into non-empty sets of size $\le 2$ such that for every $k\in \N$, 
 \[
 J^{(j+1)}_k \subset K_x \cap I_b \qquad\text{for some }x\in X', b\in \{0,1\}^j.
 \]
Similar to the base case, this is possible if $\#( K_x \cap I_b )$ is even. Otherwise, one of the sets $ J^{(j+1)}_k$ is a singleton. We continue this way until $j\in \N$ such that $2^j \ge C2^{N-\eta}$. The above construction yields that
\[
\# (I_b \cap K_x) \le 1 \qquad\text{for all } b\in \{0,1\}^j, x\in X'.
\]

The second stage starts at the level $j_0 \in \N$, which is the smallest number such that $2^{j_0} \ge C2^{N-\eta}$. Fix $b\in \{0,1\}^{j_0}$. Define
\[
m:= \sup_{x\in I_{b}} \#( I_{b} \cap B(x,r)).
\]
We claim that $m \le C^6$. Indeed, fix $x\in I_{b}$. For any $y\in I_{b} \cap B(x,r)$, let $\tilde y \in X'$ be such that $y\in K_{\tilde y}$. Observe that
\[
d(\tilde x,\tilde y) \le d(\tilde x, x) + d(x,y) + d(y, \tilde y) < 5r.
\]
Hence,  for any $y\in I_{b} \cap B(x,r)$, we have $B(\tilde y,r) \subset B(\tilde x,6r)$. By the doubling property, for any $y\in I_{b} \cap B(x,r)$ we have
\[
\# B(\tilde x,r) \le \# B(\tilde y, 6r) \le \# B(\tilde y, 8r) \le C^3 \# B(\tilde y, r).
\]
We also have
\[
\bigcup_{y\in I_{b} \cap B(x,r)} B(\tilde y,r) \subset B (\tilde x, 6r).
\]
The sets $B(\tilde y,r)$ above are disjoint. Hence,
\[
\#( I_{b} \cap B(x,r)) \frac{\# B(\tilde x,r)}{C^3} \le  \# B (\tilde x, 6r) \le C^3 \# B(\tilde x,r).
\]
This shows that $m \le C^6$.

Let $\{J^{(j_0+1)}_k\}_{k\in \N}$ be a partition of $I_{b}$, which is defined as follows. We choose sets $J^{(j_0+1)}_{k}$ to consist of pairs $x\ne y \in I_{b}$ such that  $d(x,y)<r$ and
 \[
 \#( I_{b} \cap B(x,r)) =m \quad\text{or}\quad \#( I_{b} \cap B(y,r))=m.
 \]
 This is done until we exhaust all such possible matchings of proximate elements. The remaining elements are matched in any way. If we are left with an odd number of points, then as before we add a singleton to our partition.  As a result, any choice of binary selectors $I_{b0}$ and $I_{b1}$ results in a lower value of $m$.
That is,
\[
\sup_{x\in I_{b0}} \#( I_{b0} \cap B(x,r)), \sup_{x\in I_{b1}} \#( I_{b1} \cap B(x,r)) \le m-1.
\]
Repeating this procedure $m-1$ times yields binary selectors $I_{b'}$, where $b'$ is a binary string starting with $b$ and appended by $m-1$ binary digits, such that 
\[
I_{b'} \cap B(x,r) = \{x\} 
\qquad\text{for all } x \in I_{b'}.
\]
Consequently, we are guaranteed to achieve the desired conclusion
\[
\inf \{ d(x,y): x, y \in I_{b'}, \ x\ne y \} \ge r \qquad\text{for all } {b'} \in \{0,1\}^{j}.
\]
at the level $j = j_0+m-1 \le j_0 + \lfloor C^6 \rfloor$. By the minimality of $j_0$ we have
\[
2^j = 2^{j_0-1} 2^{m} \le C2^{N-\eta+1} 2^{\lfloor C^6 \rfloor}.
\]
Choosing sufficiently large $\eta$ guarantees that $2^j \le 2^N$ and hence $j\le N$, completing the proof.
\end{proof}

\begin{theorem}\label{sse}
Let $X$ be a discrete metric space $(X,d)$ satisfying the doubling condition \eqref{doubling}. There exists a constant $\tilde c>0$ depending only on a doubling constant of $X$  such that the following holds. Suppose that $\{T_x\}_{x\in X}$ is a family of positive trace class operators in a separable Hilbert space $\mathcal{H}$ satisfying for some $\delta>0$,
\begin{equation}\label{tx1}
T:= \sum_{x\in X} T_x \le \bI \quad\text{and}\quad \tr(T_x)\le \delta \qquad\text{for all }x \in X.
\end{equation}
Let $0<\ve <1$ and $r>0$ be such that
\begin{equation}\label{tx2}
\sup_{x\in X} \# B(x,r) \le \tilde c \frac{\ve^2}{\delta}.
\end{equation}
Then, for some $N\in \N$ there exists a partition  of $X$ into uniformly discrete sets $\{I_b:b\in \{0,1\}^{N}\}$ satisfying \eqref{sepa} such that
\begin{equation}\label{tx3}
\bigg\| 2^N \sum_{x\in  I_b} T_x - T \bigg\| \le \ve \qquad\text{for all }b\in \{0,1\}^{N}.
\end{equation}
\end{theorem}

\begin{proof}
 Let $C \ge 1$ and $\eta \in \N$ be constants as in Lemma \ref{numer} and Lemma \ref{metr}, respectively. Choosing sufficiently small $\tilde c>0$, say $\tilde c=2^{-\eta-1}/C^2$, guarantees the existence of $N\in \N$ such that
 \begin{equation}\label{tx4}
 2^\eta \tilde c \frac{\ve^2}{\delta} \le 2^N \le \frac{1}{C^2} \frac{\ve^2}{\delta} \qquad\text{and}\qquad
 2^N < \frac{1}{\delta}.
 \end{equation}
Lemma \ref{metr} proves the existence of a choice of consecutive partitions of $X$ into non-empty sets of size $\le 2$, such that every binary selector of order $N$, is uniformly discrete. On the other hand, for this choice of consecutive partitions Theorem \ref{ks2i} shows the existence of selectors $I_b$,  $b\in \{0,1\}^{N}$, such that
\[
\bigg\| 2^N \sum_{x\in  I_b} T_x - T \bigg\| \le  C \sqrt{2^N \delta} \le \ve.
\qedhere
\]
\end{proof}

Specializing Theorem \ref{sse} to the case of Parseval frames we obtain the following corollary.

\begin{corollary}\label{mpa}
Let $X$ be a discrete metric space $(X,d)$ satisfying the doubling condition \eqref{doubling}. 
Let $M$ be at most countable set. Let $0<\epsilon_j<1$ be such that 
 \begin{equation}\label{mpar1}
\delta_0:= \sum_{j\in M} (1-\epsilon_j)<\infty.
 \end{equation}
 Suppose that $\{u^{(j)}_x\}_{x\in X}$ is a Parseval frame, in a Hilbert space $\mathcal H_j$, $j\in M$, such that
\begin{equation*}
\|u^{(j)}_x\|^2 \ge \epsilon_j 
\qquad\text{for all }x \in X, j\in M.
\end{equation*}
Let $r>0$ be such that
\begin{equation}\label{mpa2}
\sup_{x\in X} \# B(x,r) \le \frac{\tilde c}{4 \delta_0},
\end{equation}
where $\tilde c$ is as in Theorem \ref{sse}.
Then there exists a uniformly discrete set $X' \subset X$ satisfying 
\begin{equation}\label{mpa3}
\inf\{|x-y|: x,y \in X', x \ne y\} \ge r,
\end{equation}
 such that for every $j\in M$, $\{u^{(j)}_x\}_{x\in X \setminus X'}$ is a Riesz sequence in $\mathcal H_j$ with a lower Riesz bound $\ge 2C^2 \delta_0$, where $C$ is the same constant as in Lemma \ref{numer}.
\end{corollary}

Corollary \ref{mpa} applies only for small values of the parameter $\delta_0$. Indeed, the assumption \eqref{mpa2} forces $\delta_0\le \tilde c/4$ since $\# B(x,r) \ge 1$.

\begin{proof}
By Lemma \ref{nc} for each $j\in M$, there exists a Parseval $\{ v_x^{(j)} \} _{x\in X }$ satisfying $||v^{(j)}_x||^{2}\leq 1-\epsilon_j$ for all $x\in X$, such that an improved Bessel bound on a subsequence $\{ v_x^{(j)} \} _{x\in J }$, where $J\subset X$, implies a lower Riesz bound on $\{ u_x^{(j)} \} _{x\in J }$, and vice versa. For $x\in X$ and $j\in M$, let $T^{(j)}_x = v_x^{(j)} \otimes v_x^{(j)}$ be rank 1 operator on $\mathcal H_j$. Let $T_x=\bigoplus_{j\in M} T^{(j)}_x $ be a block diagonal trace class operator on $\bigoplus_{j\in M} \mathcal H_j$. Note that 
\[
\sum_{x\in X} T_x = \bI \quad\text{and}\quad \tr(T_x)\le \delta_0 \qquad\text{for all }x \in X
\]
Applying Theorem \ref{sse} with $\ve=1/2$ yields a uniformly discrete set $X'$ satisfying \eqref{mpa3} such that 
\[
\bigg\| 2^N \sum_{x\in  X'} T_x - \mathbf I \bigg\| \le 1/2.
\]
Since $N$ satisfies \eqref{tx4} we have
\[
\sum_{x\in  X'} T_x \ge 2^{-N-1} \mathbf I \ge 2C^2 \delta_0 \mathbf I.
\]
Hence, for every $j\in M$, the frame operator of $\{ v_x^{(j)} \} _{x\in X' }$ has lower bound $\ge 2C^2 \delta_0$. Thus, the Bessel bound of $\{ v_x^{(j)} \} _{x\in X \setminus X' }$ is $\le 1-2C^2 \delta_0$. By Lemma \ref{nc} the lower Riesz bound of $\{ u_x^{(j)} \} _{x\in X \setminus X' }$ is $\ge 2C^2 \delta_0$.
\end{proof}

\section{Discretization of continuous frames} \label{S7}

In this section we show another application of Theorem \ref{ks2i} involving a discretization of positive trace operator valued measures. In the special case when the measure takes values in rank one operators, Theorem \ref{db} recovers the solution of the discretization problem for continuous frames, which was posed by Ali, Antoine, and Gazeau \cite{aag2} and resolved by Freeman and Speegle \cite{FS}. 

In contrast to \cite{FS} we show the existence of a sampling function such that the resulting frame has nearly the same frame operator as the original continuous frame modulo a multiplicative constant. In particular, a continuous Parseval frame can be sampled to obtain a discrete frame which is nearly tight. That is, the ratio of the upper and lower bounds can be made arbitrarily close to $1$, whereas existing techniques could only guarantee this ratio to be close to $2$. In addition, our approach yields a simpler proof of the discretization problem since Theorem \ref{ks2i} provides a much easier method of controlling lower frame bounds, which was the main challenge in \cite{FS}. This is possible due to the following sampling result generalizing the corresponding result  \cite[Theorem 5.4]{FS} for scalable frames \cite{kopt}.

\begin{theorem}\label{scal} Let $I$ be countable and $\delta>0$. Suppose that $\{T_i\}_{i\in I}$ is a family of positive trace class operators in a separable Hilbert space $\mathcal{H}$ satisfying 
\begin{equation}\label{scal1}
\tr(T_i)\le \delta \qquad\text{for all }i \in I.
\end{equation}
Let $\{a_i\}_{i\in I}$ be a sequence of positive numbers such that
\begin{equation}\label{scal0}
T:= \sum_{i\in I} a_i T_i 
\end{equation}
converges to a bounded operator.
Then for any $0<\ve<1$, there exists a subset $N\subset \N$ and a sampling function $\pi: N \to I$ such that 
\begin{equation}\label{scal2}
\bigg\| \frac 1a \sum_{n\in N} T_{\pi(n)} -  T \bigg\| < \ve,
\end{equation}
for some constant $a \asymp \delta ||T|| /\ve^2$. More precisely, there exists an absolute constant $c_0>0$ such that
\begin{equation}\label{scal3}
c_0 \frac{\delta ||T||}{\ve^2} \le a \le  2 c_0 \frac{\delta ||T||}{\ve^2}.
\end{equation}
\end{theorem}

Note that the sampling function $\pi$ does not have to be $1$-to-$1$. In the proof we shall employ an elementary Hilbert space result.

\begin{lemma}\label{bfa}
Let $T$ be a positive operator on a Hilbert space $\mathcal H$. For a given subspace $\mathcal K \subset \mathcal H$, let $P_{\mathcal K}$ be the orthogonal projection of $\mathcal H$ onto $\mathcal K$. Define
\[
\gamma_1= ||P_{\mathcal K}T P_{\mathcal K}||,
\qquad
\gamma_2= ||P_{\mathcal K^\perp}T P_{\mathcal K^\perp}||.
\]
Then,
\begin{equation}\label{bfa1}
- \sqrt{\gamma_1\gamma_2} \mathbf I 
\le T - (P_{\mathcal K} T P_{\mathcal K} +P_{\mathcal K^\perp} T P_{\mathcal K^\perp} ) 
\le \sqrt{\gamma_1\gamma_2} \mathbf I .
\end{equation}
\end{lemma}

\begin{proof}
A vector $v\in \mathcal H$ decomposes as $v=v_1+v_2$, where $v_1=P_{\mathcal K} v \in \mathcal K$ and $v_2=P_{\mathcal K^\perp} v \in \mathcal K^\perp$. An elementary calculation shows that
\[
\lan (T - (P_{\mathcal K} T P_{\mathcal K} +P_{\mathcal K^\perp} T P_{\mathcal K^\perp} )) v, v \ran = \lan Tv,v \ran - \lan Tv_1,v_1 \ran - \lan Tv_2, v_2 \ran = \lan Tv_1, v_2 \ran + \lan T v_2, v_1 \ran.
\]
Since $T\ge 0$ is self-adjoint we have
\[
\gamma_1= \sup_{v\in \mathcal H, \ ||v||=1} \lan P_{\mathcal K}T P_{\mathcal K} v, v \ran
= \sup_{v\in \mathcal H, \ ||v||=1} ||T^{1/2} P_{\mathcal K} v||^2 = \sup_{v_1 \in \mathcal K,\ ||v_1||=1} ||T^{1/2} v_1||^2.
\]
A similar identity holds for $\gamma_2$. Hence,
\[
\begin{aligned}
| \lan Tv_1, v_2 \ran + \lan T v_2, v_1 \ran| \le 2 |\lan T^{1/2}v_1, T^{1/2} v_2\ran | & \le 2 \sqrt{\gamma_1 \gamma_2 } ||v_1|| ||v_2||
\\
& \le  \sqrt{\gamma_1\gamma_2} ( ||v_1||^2+ ||v_2||^2) = \sqrt{\gamma_1\gamma_2} ||v||^2.  \end{aligned}
\]
This yields \eqref{bfa1}.
\end{proof}

Next we show a special case of Theorem \ref{scal} for a finite family albeit with an improved bound on the sampling operator.

\begin{lemma}\label{scaf} Let $I$ be finite and $\delta>0$. Suppose that $\{T_i\}_{i\in I}$ is a family of positive trace class operators in a separable Hilbert space $\mathcal{H}$ and $\{r_i\}_{i\in I}$ is a sequence of natural numbers such that
\begin{equation}\label{scaf1}
T:= \sum_{i\in I} 2^{-r_i} T_i \le \bI \quad\text{and}\quad \tr(T_i)\le \delta \qquad\text{for all }i \in I.
\end{equation}
Suppose also that for some subspace $\mathcal K \subset \mathcal H$, we have
\begin{equation}\label{scaf2}
\gamma:=\tr(P_{\mathcal K} T P_{\mathcal K})  \le1.
\end{equation}
Then for any $0<\ve<1$, there exists a finite set $I'$ and a sampling function $\kappa: I' \to I$ such that 
\begin{equation}\label{scaf3}
-\ve P_{\mathcal K^\perp} - 4\sqrt{\gamma} \mathbf I 
\le \frac 1a \sum_{n\in I'} T_{\kappa(n)} - T \le
\ve P_{\mathcal K^\perp}+ 4\sqrt{\gamma} \mathbf I ,
\end{equation}
for some constant $a \asymp \delta/\ve^2$ depending only on $\delta$ and $\ve$. That is, for an absolute constant $c_0>0$ we have
\begin{equation}\label{scal3a}
c_0 \frac{\delta}{\ve^2} \le a \le  2 c_0 \frac{\delta}{\ve^2}.
\end{equation}
\end{lemma}

\begin{proof}
First, we will reduce to the case when all numbers $r_i$ are equal. Let $r \ge \sup_{i\in I} r_i$ be sufficiently large natural number to be determined later. More precisely, we need to impose that  $2^r>C^2\delta/\ve^2$, where $C$ is a constant as in Theorem \ref{ks2i}, see \eqref{scaf4}.
We replace each operator $2^{-r_i} T_i$ by a finite collection of operators
\[
\genfrac{}{}{0pt}{0}{\underbrace{2^{-r} T_i, \ldots, 2^{-r} T_i}}{m_i} \qquad\text{where } m_i:= 2^{r-r_i}.
\]
More precisely, let $m= \sum_{i\in I} m_i$ and let $\kappa: [m] \to I$ be a mapping such that each value $i\in I$ is taken precisely $m_i$ times.
This yields a new family $\{T_{\kappa(n)}\}_{n\in [m]}$, in which each operator $T_i$ is repeated $m_i$ times, and a constant sequence $\{2^{-r}\}_{n\in [m]}$. By our construction, the operator $T$ in \eqref{scal0} corresponding to the family $\{2^{-r_i} T_i \}_{i\in I}$ is the same as that of $\{2^{-r} T_{\kappa(n)} \}_{n\in [m]}$. 

Next we apply Theorem \ref{ks2i} to the family $\{2^{-r} T_{\kappa(n)} \}_{n\in [m]}$. Recall that $0<\ve<1 < C$, where the constant $C$ is as in Theorem \ref{ks2i}. Note that the parameter $N\in \N$ needs to satisfy $2^N< 2^r/\delta$. 
Thus, we can choose $N\in \N$ such that 
\begin{equation}\label{scaf4}
2^N<\frac{2^r\ve^2 }{C^2\delta} \le 2^{N+1}.
\end{equation}
Theorem \ref{ks2i} yields binary selectors $I_b$,  $b\in \{0,1\}^{N}$, that form a partition of $[m]$, such that
\begin{equation*}\label{scal5}
\bigg\| 2^{N-r} \sum_{n\in  I_b} T_{\kappa(n)} - 2^{-r}\sum_{n\in  [m]} T_{\kappa(n)}  \bigg\| \le  C \sqrt{2^{N-r} \delta} <\ve \qquad\text{for all }b\in \{0,1\}^{N}.
\end{equation*}
Thus, letting $a= 2^{r-N}$ yields 
\begin{equation}\label{scaf6}
\bigg\| \frac{1}{a} \sum_{n\in  I_b} T_{\kappa(n)} - T  \bigg\| <\ve \qquad\text{for all }b\in \{0,1\}^{N}.
\end{equation}
Moreover, \eqref{scaf4} implies \eqref{scal3a}.
Since selectors $I_b$, $b\in \{0,1\}^{N}$ form a partition of $[m]$ we have
\[
 \frac{1}{2^N a}\sum_{b\in \{0,1\}^{N}} \sum_{n\in  I_b} T_{\kappa(n)} = T.
\]
Thus, there exists $b\in \{0,1\}^N$ such that 
\[
\bigg\| P_{\mathcal K} \bigg( \frac 1a \sum_{n\in  I_b} T_{\kappa(n)} \bigg) P_{\mathcal K}\bigg\|
\le
\tr\bigg(P_{\mathcal K} \bigg(\frac 1a \sum_{n\in  I_b} T_{\kappa(n)} \bigg) P_{\mathcal K} \bigg) \le \tr(P_{\mathcal K} T P_{\mathcal K}) = \gamma.
\]
By \eqref{scaf1} and \eqref{scaf6}
\[
\bigg\| P_{\mathcal K^\perp} \bigg( \frac 1a \sum_{n\in  I_b} T_{\kappa(n)} \bigg) P_{\mathcal K^\perp }\bigg\| \le 1+\ve < 2.
\]
By Lemma \ref{bfa}
\begin{equation}\label{scaf11}
- \sqrt{2\gamma} \mathbf I 
\le  \frac 1a \sum_{n\in  I_b} T_{\kappa(n)} -  P_{\mathcal K} \bigg( \frac 1a \sum_{n\in  I_b} T_{\kappa(n)} \bigg) P_{\mathcal K} -P_{\mathcal K^\perp} \bigg( \frac 1a \sum_{n\in  I_b} T_{\kappa(n)} \bigg) P_{\mathcal K^\perp} 
\le \sqrt{2\gamma} \mathbf I .
\end{equation}
Likewise, by  Lemma \ref{bfa}
\begin{equation}\label{scaf12}
- \sqrt{\gamma} \mathbf I 
\le T - P_{\mathcal K} T P_{\mathcal K}  - P_{\mathcal K^\perp} T P_{\mathcal K^\perp} 
\le \sqrt{\gamma} \mathbf I .
\end{equation}
By \eqref{scaf6}
\begin{equation}\label{scaf13}
-\ve P_{\mathcal K^\perp}  \le
P_{\mathcal K^\perp} \bigg( \frac{1}{a} \sum_{n\in  I_b} T_{\kappa(n)} - T \bigg) P_{\mathcal K^\perp} 
\le \ve P_{\mathcal K^\perp} .
\end{equation}
Likewise,
\begin{equation}\label{scaf14}
-\gamma P_{\mathcal K}  \le
P_{\mathcal K} \bigg( \frac{1}{a} \sum_{n\in  I_b} T_{\kappa(n)} - T \bigg) P_{\mathcal K} 
\le \gamma P_{\mathcal K}.
\end{equation}
Combining \eqref{scaf11}--\eqref{scaf14} yields
\[
 -(1+\sqrt{2})\sqrt{\gamma} \mathbf I - \gamma P_{\mathcal K} - \ve P_{\mathcal K^\perp}
\le \frac 1a \sum_{n\in  I_b} T_{\kappa(n)}- T \le (1+\sqrt{2})\sqrt{\gamma} \mathbf I + \gamma P_{\mathcal K} + \ve P_{\mathcal K^\perp} .
\]
Since $\gamma \le 1$, this yields \eqref{scaf3}.
\end{proof}

We are now ready to prove Theorem \ref{scal}.

\begin{proof}[Proof of Theorem \ref{scal}]
We shall prove this result by progressively relaxing auxiliary assumptions on operators $T_i$ and coefficients $a_i$.

{\bf Step 1.}
First, we will show the special case when:
\begin{itemize}
\item all coefficients $a_i$ are of the form $a_i=2^{-r_i}$ for some $r_i\in \N$, $i\in I$,
\item all operators $T_i$ are finite rank, and
\item the operator $T$ is a contraction; that is, $T \le \mathbf I$.
\end{itemize}
Without loss of generality we can assume that the closure of the span of ranges of the operators $T_i$, $i\in I$, equals $\mathcal H$. We shall also assume that $\mathcal H$ is infinite dimensional. The finite dimensional case requires an easy modification of the following proof, which yields a sampling function $\pi$ defined on a finite subset $N \subset \N$.
By reindexing we can also assume that the index set $I=\N$.
Let $\{\gamma_k\}_{k\in\N}$ be a sequence of positive numbers defined by $\gamma_k=\ve^2 4^{-k}$, $k\in \N$. 

We shall construct an increasing sequence of natural numbers $\{K_i\}_{i\in \N}$ and a sequence of orthogonal finite dimensional spaces $\{\mathcal H_k\}_{k\in \N}$ such that $\bigoplus_{k\in \N} \mathcal H_k = \mathcal H$, by the following inductive procedure. Let $\mathcal H_1=\{0\}$ be the trivial space and $K_1=1$.  Assume we have already constructed subspaces $\mathcal H_1,\ldots,\mathcal H_n$ and natural numbers $K_1,\ldots,K_{n}$, $n\ge 1$. 
Define a finite dimensional subspace
\begin{equation}\label{hn1}
\mathcal H_{n+1} = \spa \{ P_{(\mathcal H_1  \oplus \ldots \oplus \mathcal H_n )^\perp} T_i(\mathcal H) : 1 \le i \le K_n\},
\end{equation}
Then, choose $K_{n+1}>K_n\in \N$ large enough so that
\begin{equation}\label{hn2}
\tr\bigg( P_{\mathcal H_1 \oplus \ldots \oplus \mathcal H_{n+1}}\bigg(\sum_{i>K_{n+1}} 2^{-r_i} T_i \bigg) P_{\mathcal H_1 \oplus \ldots \oplus \mathcal H_{n+1}}  \bigg) \le \gamma_{n+1}.
\end{equation}
This is possible since spaces $\mathcal H_n$ are finite dimensional and the series defining $T$ in \eqref{scal0} converges in the strong operator topology. Hence, for any finite dimensional subspace $\mathcal K \subset \mathcal H$,  the series $\sum_{i\in \N} a_i P_{\mathcal K} T_i P_{\mathcal K}$ converges in operator norm. For convenience we let $K_0=0$.

For any $n\ge 0$, we apply Lemma \ref{scaf} for a finite family $\{2^{-r_i} T_i \}_{i=K_n+1}^{K_{n+1}}$ and the subspace 
\[
\mathcal K = (\mathcal H_{n+1} \oplus \mathcal H_{n+2})^\perp =  \mathcal H_1 \oplus \ldots \oplus \mathcal H_n\oplus \bigoplus_{k \ge n+3} \mathcal H_k.
\]
By \eqref{hn1}
\[
T_i(\mathcal H) \subset \mathcal H_1 \oplus \ldots \oplus \mathcal H_{n+2} \qquad\text{for } 1\le i \le K_{n+1}.
\]
Hence, by \eqref{hn2} for $n\ge 1$ we have
\[
\begin{aligned}
\tr\bigg( P_{\mathcal K}\bigg(\sum_{i=K_{n}+1}^{K_{n+1}} 2^{-r_i} T_i \bigg) P_{\mathcal K}  \bigg) 
&=
\tr\bigg( P_{\mathcal H_1 \oplus \ldots \oplus \mathcal H_{n}} \bigg(\sum_{i=K_{n}+1}^{K_{n+1}} 2^{-r_i} T_i \bigg) 
P_{\mathcal H_1 \oplus \ldots \oplus \mathcal H_{n}} \bigg) 
\\
& \le
\tr\bigg( P_{\mathcal H_1 \oplus \ldots \oplus \mathcal H_{n}}\bigg(\sum_{i=K_{n}+1}^\infty 2^{-r_i} T_i \bigg) P_{\mathcal H_1 \oplus \ldots \oplus \mathcal H_{n}}  \bigg) \le \gamma_{n}.
\end{aligned}
\]
The same bound also holds trivially for $n=0$ with $\gamma_0=0$.
By Lemma \ref{scaf} there exists a finite set $I_n$ and a sampling function $\kappa_n: I_n \to (K_n+1,K_{n+1}]\cap \N$ such that
\begin{equation}\label{hn3}
-\ve P_{\mathcal H_{n+1} \oplus \mathcal H_{n+2}} - 4\sqrt{\gamma_n} \mathbf I 
\le \frac 1a \sum_{i\in I_n} T_{\kappa_n(i)} - \sum_{i=K_{n}+1}^{K_{n+1}} 2^{-r_i} T_i \le
\ve P_{\mathcal H_{n+1} \oplus \mathcal H_{n+2}}+ 4\sqrt{\gamma_n} \mathbf I .
\end{equation}
Recall that the constant $a$ satisfies \eqref{scal3a}, depends only on $\delta$ and $\ve$,  and is independent of $n$. Also note that
\[
 \sum_{n=0}^\infty P_{\mathcal H_{n+1} \oplus \mathcal H_{n+2}} = 2 \mathbf I,
\]
where the sum converges in the strong operator topology. Since $\gamma_0=0$ we have
\[
4 \sum_{n=0}^\infty \sqrt{\gamma_n} = 4 \sum_{n=1}^\infty \ve 2^{-n} = 4 \ve.
\]
Hence, summing \eqref{hn3} over $n\ge 0$  yields
\[
-6 \ve \mathbf I \le  
\frac 1a \sum_{n=0}^\infty \sum_{i\in I_n} T_{\kappa_n(i)} - \sum_{i\in \N} 2^{-r_i} T_i 
\le 6 \ve \mathbf I.
\]
Gluing sampling functions on the disjoint union of sets $I_n$, $n\ge 0$, yields after suitable reindexing a sampling function $\pi: \N \to I$ such that
\[
-6 \ve \mathbf I \le  
\frac 1a \sum_{i\in \N} T_{\pi(i)} - T 
\le 6 \ve \mathbf I.
\]
After rescaling this yields the conclusion \eqref{scal2}.

{\bf Step 2.} Next, we relax the assumption about coefficients $a_i$ by representing each $a_i$ in binary form as $a_i=\sum_{k\in \N} 2^{-r^{(i)}_k}$ for some sequence $\{r^{(i)}_k\}_{k\in \N}$ of natural numbers. Then, we apply Step 1 to the family $\{2^{-r^{(i)}_k} T_i\}_{i\in I, k\in \N}$. Since the corresponding operator $T$ in \eqref{scal0} stays the same, we obtain the required conclusion \eqref{scal2}.

{\bf Step 3.} Subsequently, we relax that assumption that operators $T_i$ have finite rank.
Choose a sequence $\{\epsilon_i\}_{i\in I}$ of positive numbers such that $\sum_{i\in I} \epsilon_i < \ve$. For $i\in I$, let $T_i'$ be a finite rank operator truncating all but a finite number of largest eigenvalues and eigenvectors with $T_i$ such that $0 \le T_i' \le T_i$ and
\begin{equation}\label{hn5}
||T_i-T'_i||<  ||T_i||  \epsilon_i.
\end{equation}
Note that
\[
T':= \sum_{i\in I} a_i T'_i \le \bI \quad\text{and}\quad \tr(T'_i)\le \delta \qquad\text{for all }i \in I.
\]
By Step 2 there exists a sampling function $\pi: \N \to I$ such that 
\begin{equation}\label{hn6}
\bigg\| \frac 1{a} \sum_{n\in \N} T'_{\pi(n)} -  T' \bigg\| < \ve,
\end{equation}
where $a \asymp \delta/\ve^2$ is a constant satisfying \eqref{scal3a}.
For any $i\in I$,
\[
\# \{n\in \N: \pi(n)=i \}||T_i'|| = \bigg\| \sum_{n\in \N, \ \pi(n)=i } T'_{\pi(n)} \bigg\|
 \le \bigg\| \sum_{n\in \N} T'_{\pi(n)} \bigg\| \le a(1+\ve).
 \]
Thus,
\begin{equation}\label{hn7}
\bigg\| \frac 1{a} \sum_{n\in \N} T'_{\pi(n)} -  \frac 1{a} \sum_{n\in \N} T_{\pi(n)} \bigg\|
\le \frac{1}{a} \sum_{i\in I} \sum_{n\in \N: \pi(n)=i } ||T'_i - T_i|| \le \sum_{i\in I} (1+\ve) \ve_i< 2\ve.
\end{equation}
Since $a_i ||T_i||\le 1$, we also have by \eqref{hn6} 
\begin{equation}\label{hn8}
||T - T'|| \le \sum_{i\in I} a_i ||T_i - T_i'|| \le \sum_{i\in I} \ve_i < \ve.
\end{equation}
Combining \eqref{hn6}--\eqref{hn8} yields
\[
\bigg\| \frac 1{a} \sum_{n\in \N} T_{\pi(n)} -  T \bigg\|  < 4\ve.
\]

{\bf Step 4.} Finally, we relax the assumption that $T$ is a contraction. Without loss of generality we can assume that $||T||>1$. Applying Step 3 to the normalized coefficients $\{a_i/||T||\}_{\in I}$ yields a sampling function $\pi: \N \to I$ such that 
\begin{equation}\label{hn6b}
\bigg\| \frac 1{a'} \sum_{n\in \N} T_{\pi(n)} -  \frac{T}{||T||} \bigg\| < \frac{\ve}{||T||},
\end{equation}
where $a' \asymp \delta ||T||^2 /\ve^2$ is a constant. Multiplying \eqref{hn6b} by $||T||$ yields \eqref{scal2} with $a=a'/||T|| \asymp \delta ||T|| /\ve^2$ satisfying \eqref{scal3}.
\end{proof}

We are now ready to show the discretization problem for trace class positive operator valued measures (POVMs). We recall the definition of  compact operator-valued Bessel family in \cite[Section 4]{MB2}.

\begin{definition}\label{cov}
Let $K_+(\mathcal H)$ be the space of positive compact operators on a separable Hilbert space $\mathcal H$. Let $(X, \mu)$ be a measure space. We say that $T = \{T_t\}_{t\in X}: X \to K_+(\mathcal H)$ is  {\it compact operator-valued Bessel family} if:
\begin{enumerate}
\item 
for each $f,g \in\mathcal H$, the function $X \ni t \to \lan T_t f, g \ran \in \C$ is measurable, and 
\item
there exists a constant $B>0$ such that
\[
\int_X \lan T_t f,f \ran d\mu(t) \le B ||f||^2 \qquad \text{for all }f\in \mathcal H.
\]
\end{enumerate}
\end{definition}

We need the following approximation result for compact POVMs generalizing rank one result \cite[Lemma 1]{MB3}.

\begin{lemma}\label{approx}
Let $(X,\mu)$ be a measure space and let $\mathcal H$ be a separable Hilbert space.
Suppose that $\{T_t\}_{t\in X}$ is a compact operator-valued Bessel family in $\mathcal H$. 
Define a frame operator $\mathcal S_{T}$ on $\mathcal H$ by
\begin{equation}\label{cov5}
\mathcal S_{T} f = \int_X T_t f d\mu(t) \qquad\text{for }f\in \mathcal H.
\end{equation}
Then, for every $\ve>0$,
\begin{enumerate}
\item there exists a compact operator-valued Bessel family $\{R_t\}_{t\in X}$, 
which takes only countably many values, such that
\begin{equation}\label{ap2}
||\mathcal S_T - \mathcal S_R ||<\ve,
\end{equation}
\item
there exists a partition $\{X_n\}_{n\in \N}$ of $X$ into measurable sets and a sequence $\{t_n\}_{n\in\N}$ in $X$, such that $t_n \in X_n$ and
\begin{equation}\label{ap1}
R_t = T_{t_n} \qquad\text{for a.e. }t \in X_n, \ n\in\N.
\end{equation}
\end{enumerate}
\end{lemma}

\begin{proof}
By \cite[Proposition 2.1 and Remark 4.1]{MB2}, the support $\{t\in X: T_t \not =0\}$ is a $\sigma$-finite subset of $X$. Hence, by restricting to the support, we can assume that measure space $X$ is $\sigma$-finite. The space $X$ can be decomposed into its atomic $X_{at}$ and non-atomic $X \setminus X_{at}$ parts. Since $X$ is $\sigma$-finite, it has at most countably many atoms. Since every measurable mapping is constant a.e. on atoms, we can take $R_t = T_t$ for all $t\in X_{at}$, and the conclusions (i) and (ii) hold on $X_{at}$. Therefore, without loss of generality we can assume that $\mu$ is a non-atomic measure.
(Note that when the space $X$ has only finitely many atoms and does not have non-atomic part, then the resulting partition in (ii) is necessarily finite. With this modification, the lemma holds trivially.)

Since the space $K(\mathcal H)$ of compact operators on $\mathcal H$ is separable, by the Pettis Measurability Theorem \cite[Theorem II.2]{DU}, the weak measurability (i) in Definition \ref{cov} is equivalent to (Bochner) strong measurability on $\sigma$-finite measure space $X$. That is, $t \mapsto T_t$ is a pointwise a.e. limit of simple measurable functions. Moreover, by \cite[Corollary II.3]{DU}, every measurable function $T: X\to K(\mathcal H)$ is a.e. uniform limit of a sequence of countably-valued measurable functions. Although these results were stated in \cite{DU} for finite measure spaces, they also hold for $\sigma$-finite measure spaces.

Define measurable sets $Y_0=\{t\in X: ||T_t||<1\}$ and
\[
Y_n= \{t\in X: 2^{n-1} \le ||T_t||< 2^n \}, \qquad n\ge 1.
\]
Then, for any $\ve>0$, we can find a partition $\{Y_{n,m}\}_{m\in \N}$ of each $Y_n$ such that $\mu(Y_{n,m}) \le 1$ for all $m\in\N$. Applying \cite[Corollary II.3]{DU} to each family $\{T_t \}_{t\in Y_{n,m}}$ yields a countably-valued measurable function $\{ \tilde T_t\}_{t\in Y_{n,m}}$ such that 
\begin{equation}\label{ss1}
||\tilde T_t - T_t || \le \frac{\ve}{4^n 2^{m+1}} \qquad\text{for a.e. }t\in Y_{n,m}.
\end{equation}
Since $\{Y_{n,m}\}_{n\in \N_0, m\in \N}$ is a partition of $X$, we obtain a global countably-valued  function $\{\tilde T_t\}_{t\in X}$ satisfying \eqref{ss1}. Thus, we can partition $X$ into countable family of measurable sets $\{X_k\}_{k\in\N}$ such that $\{\tilde T_t\}_{t\in X}$ is constant on each $X_k$. Moreover, we can also require that $\{X_k\}_{k\in \N}$ is a refinement of a partition $\{Y_{n,m}\}_{n\in \N_0, m\in \N}$.

For fixed $k\in \N$,  take $n$ and $m$ such that $X_k \subset Y_{n,m}$. Choose $t_k \in X_k$ for which \eqref{ss1} holds. Define a countably-valued function $\{R_t\}_{t\in X}$ by
\[
R_t = T_{t_k} \qquad \text{for }t \in X_k, \ k\in \N.
\]
Thus, the conclusion (ii) follows by the construction.

Now fix $n\in \N_0$ and $m\in \N$, and take any $t\in Y_{n,m}$ outside the exceptional set in \eqref{ss1}. Let $k\in \N$ be such that $t\in X_k$. Since $T_t$ is constant on $X_k$, by \eqref{ss1} we have
\[
||T_t - R_t || = || T_{t} - T_{t_k}|| \le || T_t - \tilde T_t|| +  ||\tilde T_{t_k} -  T_{t_k}|| \le 2 \frac{\ve}{4^n 2^{m+1}}.
\]
Since operators $T_t$ and $R_t$ are self-adjoint, for any $f\in\mathcal H$ with $||f||=1$ we have
\[
|\langle T_t f, f \rangle  - \langle R_t f, f \rangle |
\le
||T_t - R_t || \le \frac{\ve}{4^n 2^m} 
\qquad\text{for a.e. }t\in Y_{n,m}.
\]
Integrating over $Y_{n,m}$ and summing over $n\in\N_0$ and $m\in\N$ yields
\[
\int_X |\langle T_t f, f \rangle  - \langle R_t f, f \rangle|  d\mu(t)  \le  \sum_{n=0}^\infty\sum_{m=1}^\infty \frac{\ve}{2^n2^m} \mu(Y_{n,m}) \le 2\ve.
\]
Thus,
\[
\begin{aligned}
||\mathcal S_T - \mathcal S_R || 
= \sup_{||f||=1} |\langle (\mathcal S_T - \mathcal S_R)f,f \rangle |
= \sup_{||f||=1} \bigg|\int_X ( \langle T_t f, f \rangle  - \langle R_t f, f \rangle )  d\mu(t)    \bigg| \le 2\ve.
\end{aligned}
\]
Since $\ve>0$ is arbitrary, this completes the proof.
\end{proof}

As a corollary of Theorem \ref{scal} and Lemma \ref{approx} we obtain a discretization result for trace class POVMs.

\begin{theorem}\label{db} Let $(X,\mu)$ be a measure space and let $\mathcal H$ be a separable Hilbert space. Let $\delta,\ve>0$.
Suppose that $\{T_t\}_{t\in X}$ is a trace class operator-valued Bessel family in $\mathcal H$ such that
\[
\tr(T_t) \le \delta<\infty \qquad\text{for a.e. }x\in X.
\]
Let $\mathcal S_T$ be the frame operator of $\{T_t\}_{t\in X}$ given by \eqref{cov5}.
Then, there exists a sampling sequence $\{t_n\}_{n\in I}$ in $X$, where $I\subset \N$, such that 
\begin{equation}\label{db0}
\bigg\| \frac1{a} \sum_{n\in I} T_{t_n} - \mathcal S_T \bigg\| <\ve,
\end{equation}
for some constant $a \asymp \delta ||\mathcal S_T|| /\ve^2$ satisfying the bound of the form \eqref{scal3}.
\end{theorem}

\begin{proof}
Let $B= ||\mathcal S_T||$ be the Bessel bound of $\{T_t\}_{t\in X}$.
Without loss of generality we can assume that $0<\ve\le B$. Otherwise, we can take an empty sampling sequence.
By Lemma \ref{approx} we can find a partition $\{X_n\}_{n\in\N}$ of $X$ and a sequence $\{t_n\}_{n\in \N}$ in $X$ such that 
\begin{equation}\label{db1}
\bigg\| \sum_{n\in \N} a_n T_{t_n} - \mathcal S_T \bigg\| < \ve/2,
\qquad\text{where } a_n = \mu(X_n).
\end{equation}
Hence,
\[
T:=\sum_{n\in \N} a_n T_{t_n} \le (B+\ve/2) \mathbf I \le \tfrac{3}{2}B \mathbf I.
\]
Applying Theorem \ref{scal} yields a constant $a \asymp \tfrac32 \delta B/(\ve/2)^2$ and a sampling function $\pi: \N \to \N$ such that 
\begin{equation}\label{db2}
\bigg\| \frac 1{a} \sum_{n\in \N} T_{t_{\pi(n)}} -   T \bigg\| < \ve/2.
\end{equation}
Combining \eqref{db1} and \eqref{db2} yields \eqref{db0}.
\end{proof}

In the special case of continuous frames we have the following corollary, which improves the main result of Freeman and Speegle \cite[Theorem 5.7]{FS}.

\begin{corollary}\label{cff} Let $\psi: X \to \mathcal H$ be a continuous frame such that 
\[
\| \psi(t) \|^2 \le \delta<\infty \qquad\text{for a.e. }t\in X.
\]
That is, there are constants $0<A \le B< \infty$, called frame bounds, such that 
\begin{equation}\label{cf1}
A||f||^2 \le \int_X |\langle f, \psi(t) \rangle|^2 d\mu (t) \le B ||f||^2 \qquad\text{for all }f\in\mathcal H.
\end{equation}
Then for all $0<\ve<1$, there exists a sampling sequence $\{t_n\}_{n\in I}$ in $X$ such that 
$\{\psi(t_n)\}_{n\in I}$ is a frame in $\mathcal H$ with the lower frame bound $(1-\ve)a$ and the upper frame bound $\frac{B}{A}(1+\ve)a$ for some constant $a\asymp \frac{\delta }{\ve^2}$.
\end{corollary}

\begin{proof}
Define a rank one operator-valued Bessel family $\{T_t\}_{t\in X}$ by 
\[
T_tf = \lan f, \psi(t) \ran \psi(t) \qquad\text{for } f\in \mathcal H, t\in X.
\]
Note that $\tr(T_t)= ||\psi(t)||^2 \le \delta$. Let $\mathcal S_T$ be the frame operator of a continuous frame $\psi: X \to \mathcal H$, which is given by \eqref{cov5}. The frame bounds \eqref{cf1} imply that 
\begin{equation}
\label{sab}
A\mathbf I \le \mathcal S_T \le B \mathbf I.
\end{equation} 
A canonical Parseval frame $\tilde \psi: X \to \mathcal H$ is given by $\tilde \psi(t) = (\mathcal S_T)^{-1/2}\psi(t)$ for $t\in X$, and satisfies 
\[
||\tilde \psi(t)||^2 \le ||\psi(t)||^2/A \le \delta/A\qquad\text{for a.e. }t\in X.
\]
Let $\{\tilde T_t\}_{t\in X}$ be the corresponding rank one operator-valued family given by
\[
\tilde T_tf = \lan f, \tilde \psi(t) \ran \tilde \psi(t) \qquad\text{for } f\in \mathcal H, t\in X.
\]
Then  $\tilde T_t = (\mathcal S_T)^{-1/2} T_t (\mathcal S_T)^{-1/2}$, and the frame operator of $\{\tilde T_t\}_{t\in X}$ is the identity operator $\mathbf I$.
By Theorem \ref{db} applied to $\{\tilde T_t\}_{t\in X}$, there exists a sequence $\{t_n\}_{n\in I}$ in $X$ such that 
\[
(1-\ve)\mathbf I \le \frac1{a'} \sum_{n\in I} \tilde T_{t_n}  \le (1+\ve)\mathbf I,
\]
for some constant $a' \asymp \delta  /(A\ve^2)$. Hence, by \eqref{sab} the system $\{\psi(t_n)\}_{n\in I}= \{ (\mathcal S_T)^{1/2} \tilde \psi(t_n)\}_{n\in I}$ is a frame with bounds $A(1-\ve)a'$ and  $B(1+\ve)a'$, and $\sum_{n\in I} T_{t_n}$ is its frame operator.
Consequently, the frame bounds of $\{\psi(t_n)\}_{n\in I}$ are of the form $(1-\ve)a$ and $\frac{B}{A}(1+\ve)a$, where $a=a'A \asymp \frac{\delta }{\ve^2}$.
\end{proof}

In the special case when frame bounds are tight  $A=B$, Corollary \ref{cff} shows that for any $0<\ve<1$, a continuous tight frame can be sampled to a discrete (nearly tight) frame with bounds $(1\pm \ve) a$ for some constant $a \asymp \delta /\ve^2$. It is remarkable that the bounds of a resulting discrete frame are independent of the frame constant $A=B$ of a continuous tight frame.

\section{Applications to systems of exponentials}\label{S8}

In this section we explore applications of selector results to systems of exponentials. We start by showing consequences of Corollary \ref{mfei} for exponential frames. The following result is a generalization of  a positive density result of Bourgain and Tzafriri  \cite[Theorem 2.2]{BT2} and a syndetic result of Londner and the author \cite[Corollary 2.3]{BL}.

\begin{corollary}\label{rect}
There exists a universal constant $c>0$ such that for any $\ve>0$ and subset $S\subset \T^d=(\R/\Z)^d$ of positive measure, there exists a set of frequencies $\Lambda \subset \Z^d$ such that:
\begin{itemize}
\item the exponential system  $E(\Lambda)=\{ e^{2\pi i \langle \lambda ,x \rangle } \} _{\lambda\in\Lambda}$  is a Riesz sequence in $L^{2}(S)$ with nearly tight bounds $(1\pm \ve)|S|$, and
\item every ball of radius at least $ c\sqrt{d} (|S|\ve^2)^{-1/d}$ contains a point in $\Lambda$, i.e.,
	\begin{equation}\label{cub}
	\sup_{ x\in\mathbb Z^d } \inf_{\lambda\in\Lambda} |\lambda-x|
	\leq c\sqrt{d} (|S|\ve^2)^{-1/d}.
	\end{equation}
\end{itemize}
\end{corollary}

\begin{proof} Let $\{ J_{k}\} _{k \in K}$ be any collection of disjoint subsets of $\Z^d$ such that 
\begin{equation}\label{rect2}
\# J_k \geq \frac{C}{|S|\ve^2}
\qquad\text{for all } k \in K.
\end{equation}
The exponential system $E(\Z^d)$ is a Parseval frame in $L^2(S)$.
Applying Corollary \ref{mfei} to a tight frame $|S|^{-1/2} E(\Z^d)$ with constant $|S|^{-1}$ yields a selector $\Lambda \subset\bigcup_{k}J_{k}$
such that $|S|^{-1/2} E(\Lambda)$ is a Riesz sequence with bounds $ 1-\ve$ and $1+\ve$. Hence, $E(\Lambda)$  is a Riesz sequence in $L^{2}(S)$ with bounds $(1\pm \ve)|S|$.

Take a cube $\mathcal R= [0,s)^d \cap \Z^d$ with side length $s=\lceil (C/(\ve^2|S|))^{1/d} \rceil$ and the corresponding lattice partition
\[
J_k = k + \mathcal R, \qquad k\in s\Z^d.
\]
This choice of partition yields a set $\Lambda \subset \Z^d$ satisfying the bound \eqref{cub}.
\end{proof}

Partitioning the lattice $\Z^d$ in a more complicated pattern we can deduce the following result on syndetic sections. Recall that a subset of integers 
\[
\Lambda = \{ \ldots<\lambda_{0}<\lambda_{1}<\lambda_{2}<\ldots
\} \subset\mathbb{Z}
\]
is syndetic if gaps between consecutive elements remain bounded 
\[
\gamma(\Lambda):= \sup_{n\in \Z} (\lambda_{n+1}-\lambda_{n}) <\infty.
\]

\begin{corollary}\label{lsyn}
There exists a universal constant $C>0$ such that for any $\ve>0$ and any subset $S\subset \T^{d}$ 
of positive measure, there exists a set $\Lambda \subset \Z^{d}$ so that:
\begin{itemize}
\item the exponential system  $E(\Lambda)=\{ e^{2\pi i \langle \lambda ,x \rangle } \} _{\lambda\in\Lambda}$  is a Riesz sequence in $L^{2}(S)$ with nearly tight bounds $(1\pm \ve)|S|$, and
\item  $\Lambda$ is syndetic along any 
of its one dimensional sections. 
That is, for any $j=1,\ldots,d$ and 
any $(k_{1},\ldots,\hat k_j, \ldots, k_{d})\in\Z^{d-1}$ the set
\[
\Lambda_j(k_{1},\ldots,\hat k_j, \ldots, k_{d})=\{k_j\in \Z :(k_{1},\ldots,k_j, \ldots,  k_{d} ) \in \Lambda\}
\]
is a syndetic subset of integers with gap satisfying 
\begin{equation}\label{gap}
\gamma (\Lambda_j (k_{1},\ldots,\hat k_j, \ldots, k_{d}) ) \leq C d (\ve^2|S |)^{-1}.
\end{equation}
The notation $\hat k_j$ means that the coordinate $k_j$ is missing.
\end{itemize}
\end{corollary}

Corollary \ref{lsyn} can be shown the same way as \cite[Corollary 2.4]{BL}. We leave the details of the proof to the reader. 

Next we show an application  of Corollary \ref{mpa}. Specializing it to a single Parseval frame of exponentials answers an open problem by Londner and the author in \cite{BL}.

\begin{theorem}\label{B}
There exists a dimensional constant $C=C(d)>0$ such that the following holds. 
For any measurable subset $S\subset\T^d=(\R/ \Z)^d$ with positive measure, there exists a subset $\Lambda\subset\Z^d$ so that $E(\Lambda)=\{ e^{2\pi i \langle \lambda ,x \rangle } \} _{\lambda\in\Lambda}$ is a Riesz sequence in $L^{2}(S)$ and $\Lambda^{c}=\Z^d\setminus \Lambda$ is a uniformly discrete set satisfying 
\begin{equation}\label{eqs}
		\inf_{\lambda,\mu\in\Lambda^{c},\lambda\neq\mu} |\lambda-\mu|\geq\frac{C}{| S^c |^{1/d}}
		\qquad\text{where }S^{c} =\T^d \setminus S.
\end{equation}
\end{theorem}

\begin{proof}
If $S \subset \T^d$ has a full measure, then $E(\Z^d)$ is an orthonormal basis of $L^{2}(S)$. In general, the system of exponentials $E(\Z^d)$ is a Parseval frame in $L^{2}(S)$ consisting of functions with squared norm equal to the Lebesgue measure $|S|$ of $S \subset \T^d$. Without loss of generality we can assume that $S^{c}$ has small measure since otherwise there is nothing to show.
Let $r>0$ be the largest radius such that the number of lattice points $\Z^d$ in the ball $B(0,r)$ satisfies
\[
\# (B(0,r) \cap \Z^d)  \le \frac{\tilde c}{4 |S^{c}|}.
\]
Since $r$ is large we have $\# (B(0,r)\cap \Z^d)  \asymp \operatorname{vol} B(0,r) = c_d r^d$. Thus, $r\asymp (\tilde c/(4 c_d |S^{c}|))^{1/d}$. By Corollary \ref{mpa} there exists  a uniformly discrete set $\Lambda^c \subset \Z^d $ satisfying
\[
\inf_{\lambda,\mu\in\Lambda^{c},\lambda\neq\mu} |\lambda-\mu|\geq r,
\]
and such that $E(\Lambda)$ is a Riesz sequence in $L^2(S)$. This yields \eqref{eqs}.
\end{proof}

A similar proof yields a stronger result for a family of measurable sets of nearly full measure using full strength of Corollary \ref{mpa}. Alternatively, Theorem \ref{BX} can be easily deduced from Theorem \ref{B} using de Morgan's law.  We leave the details of the proof to the reader.

\begin{theorem} \label{BX}
Let $S_1,S_2, \ldots $ be measurable subsets of  $\T^d$ with positive measure such that
\[
\sum_{n\in \N} |(S_n)^c| <\infty.
\]
Then, there exists a subset $\Lambda\subset\Z^d$ so that $E(\Lambda)=\{ e^{2\pi i \langle \lambda ,x \rangle } \} _{\lambda\in\Lambda}$ is a Riesz sequence in every space $L^{2}(S_n)$, $n\in \N$, with uniform Riesz bounds. Moreover, $\Lambda^{c}=\Z^d\setminus \Lambda$ is a uniformly discrete set satisfying 
\begin{equation*}
		\inf_{\lambda,\mu\in\Lambda^{c},\lambda\neq\mu} |\lambda-\mu|\geq C \bigg( \sum_{n\in \N} |(S_n)^c|\bigg)^{-1/d}.
\end{equation*}
\end{theorem}

Finally, we show an enhancement of Corollary \ref{cff} in the setting of continuous frames of exponentials. Corollary \ref{nou} improves upon the result of Nitzan, Olevskii, and Ulanovskii \cite{NOU} who showed the existence of exponential frames for every (unbounded) set of finite measure. We improve upon their result in two ways. We construct exponential frames that are nearly tight with an explicit control on their frame redundancy. Moreover, we show that the set of frequencies can be chosen to be uniformly discrete.

\begin{corollary} \label{nou}There exist constants $c_0, c_1>0$ such that the following holds. For any set $S \subset \R^d$ of finite measure and $\ve>0$, there exists a discrete set $\Lambda \subset \R^d$ with the property that $E(\Lambda)=\{ e^{2\pi i \langle \lambda ,x \rangle } \} _{\lambda\in\Lambda}$ is a frame in $L^2(S)$ with frame bounds $(1\pm \ve)a|S|/\ve^2$, where $c_0/2\le a \le c_0$. Moreover, $\Lambda$ is a uniformly discrete set satisfying 
\begin{equation}\label{nou1}
		\inf_{\lambda,\mu\in\Lambda,\lambda\neq\mu} |\lambda-\mu|\geq c_1 (\ve^2/ |S|)^{1/d}.
\end{equation}
\end{corollary}

\begin{proof}
Suppose first $S$ is a measurable subset of $[0,1]^d$. Since $E(\Z^d)$ is a Parseval frame in $L^2(S)$, rank one operators $e_k \otimes e_k$, where $e_k(x)=e^{2\pi i \langle \lambda ,x \rangle }$, satisfy 
\[
\sum_{k\in \Z^d} e_k \otimes e_k = \mathbf I_{L^2(S)} \qquad \tr(e_k \otimes e_k) = |S| \text{ for } k\in \Z^d.
\]
We shall apply Theorem \ref{sse}. Let $r>0$ be the largest radius such that the number of lattice points $\Z^d$ in the ball $B(0,r)$ satisfies
\[
\# (B(0,r) \cap \Z^d)  \le \frac{\tilde c\ve^2}{ |S|},
\]
where constant $\tilde c$ is as in \eqref{tx2}.
Without loss of generality, we can assume that $\ve^2/ | S|$ is large; otherwise the conclusion \eqref{nou1} is automatic.
This implies that $r$ is also large and we have $\# (B(0,r) \cap \Z^d)  \asymp \operatorname{vol} B(0,r) = c_d r^d$. Thus, $r\asymp (\tilde c \ve^2 / (c_d |S|))^{1/d}$. By Theorem \ref{sse} there exists  a uniformly discrete set $\Lambda \subset \Z^d $ satisfying
\[
\inf_{\lambda,\mu\in\Lambda,\lambda\neq\mu} |\lambda-\mu|\geq r,
\]
and a number $N\in \N$ such that
\begin{equation}\label{nou4}
\bigg\| 2^N \sum_{k\in  \Lambda} e_k \otimes e_k - \mathbf I \bigg\| \le \ve .
\end{equation}
Thus, \eqref{nou1} holds for some constant $c_1$.
Moreover, by \eqref{tx4} we can take $N$ to be the smallest number satisfying
\begin{equation}\label{nou5}
 2^\eta \tilde c \frac{\ve^2}{|S|} \le 2^N \le \frac{1}{C^2} \frac{\ve^2}{|S|}. 
 \end{equation}
 Letting $c_0 =  2^{-\eta} / \tilde c$ and $a= 2^{-N} \ve^2/|S|$, \eqref{nou4} yields
\begin{equation*}
\bigg\| \sum_{k\in  \Lambda} e_k \otimes e_k - \frac{a|S|}{\ve^2} \mathbf I \bigg\| \le \frac{a|S|}{\ve}.
\end{equation*}
We conclude that $E(\Lambda)$ is a frame in $L^2(S)$ with frame bounds $(1\pm \ve)a|S|/\ve^2$.

To extend this result to general sets $S \subset \R^d$, we follow a scheme as in the proof of \cite[Theorem 10.14]{OU}. By scaling we conclude that Corollary \ref{nou} holds for sets $S \subset [0,r]^d$ with $\Lambda \subset r^{-1/d}\Z^d$ for any $r>0$. Consequently, the required conclusion holds for any bounded set $S \subset \R^d$ of positive measure. 

Let $S_1 \subset S_2 \subset \ldots$ be a sequence of bounded sets  such that $S = \bigcup_j S_j$. For each $j\in \N$, let $\Lambda_j \subset \R^d$ be a uniformly discrete set satisfying \eqref{nou1} such that $E(\Lambda_j)$ is a frame in $L^2(S_j)$ with bound $(1\pm \ve)a_j|S_j|/\ve^2$. By choosing a subsequence, we may assume that $a_j$'s converge to some limit $c_0 \le a\le c_0$. Since each set $\Lambda_m$ is uniformly discrete, we may also assume that  as $m \to \infty$, sets $\Lambda_m$ converge weakly to some set $\Lambda$ satisfying \eqref{nou1}, see \cite[Section 3.4]{OU}. By \cite[Lemma 10.22]{OU} the frame property of exponentials is preserved under weak limits. Hence, $E(\Lambda)$ is a frame in $L^2(S_j)$ with bound $(1\pm \ve)a |S|/\ve^2$. Since $j\in \N$ is arbitrary, $E(\Lambda)$ is also a frame in $L^2(S)$ with the same bounds.
\end{proof}

\bibliographystyle{amsplain}

\end{document}